\documentclass[sn-mathphys]{sn-jnl}

\usepackage[T1]{fontenc}
\usepackage[english]{babel}
\usepackage{amsmath}
\usepackage{mathtools}
\usepackage{amsfonts}
\usepackage{amssymb}
\usepackage{gensymb,MnSymbol}
\usepackage{hyperref}
\usepackage{empheq}
\usepackage[numbers]{natbib}
\usepackage{paralist}
\usepackage{bbm}
\usepackage{url}
\usepackage{longtable}
\usepackage{program}
\catcode`\|=12\relax
\usepackage{graphicx}
\usepackage{pgfplots}
\usetikzlibrary{calc}
\usetikzlibrary{matrix}
\pgfplotsset{compat=1.16, 
	     width=0.95\textwidth, 
	     height=0.8\textwidth,
	     tick label style={font=\tiny},
	     label style={font=\footnotesize},	
	     title style = {font=\small},	
	     legend style = {font=\small},
	     grid=major,
	    }

\usepackage[%final   
            draft  
            ]{changes}
\definechangesauthor[name={Nina Beranek}, color=cyan]{NB}
\newcommand\Dt{{\Delta{t}}}

\newcommand*\diff{\mathop{}\!\mathrm{d}} %gerades d by Integralen
\newcommand*\Laplace{\mathop{}\Delta} % Laplace-Operator
\newcommand\dt{\diff t}

\renewcommand{\S}{Section }

\renewcommand{\R}{\mathbb{R}}

\newcommand{\cG}{{\mathcal{G}}}
\newcommand{\cH}{{\mathcal{H}}}

\newcommand{\cO}{{\mathcal{O}}}

\newcommand{\cU}{{\mathcal{U}}}
\newcommand{\cV}{{\mathcal{V}}}
\newcommand{\cW}{{\mathcal{W}}}
\newcommand{\cX}{{\mathcal{X}}}

\newcommand{\cY}{{\mathcal{Y}}}

\newcommand{\cZ}{{\mathcal{Z}}}

\newcommand{\bB}{{\boldsymbol{B}}}
\newcommand{\bD}{{\boldsymbol{D}}}

\newcommand{\bL}{{\boldsymbol{L}}}
\newcommand{\bM}{{\boldsymbol{M}}}

\newcommand{\bR}{{\boldsymbol{R}}}

\newcommand{\hatbJ}{{\boldsymbol{\hat{J}}}}

\newcommand{\bbf}{{\boldsymbol{f}}}
\newcommand{\bg}{{\boldsymbol{g}}}
\newcommand{\bh}{{\boldsymbol{h}}}
\newcommand{\bd}{{\boldsymbol{d}}}

\newcommand{\bu}{{\boldsymbol{u}}}

\newcommand{\by}{{\boldsymbol{y}}}
\newcommand{\bz}{{\boldsymbol{z}}}

\newcommand\ts[1]{{\textstyle{#1}}}

\newcommand{\fnameone}{F1}
\newcommand{\fnametwo}{F2}

\jyear{2022}%

\theoremstyle{thmstyleone}%
\newtheorem{theorem}{Theorem}[section]
\newtheorem{proposition}[theorem]{Proposition}

\newtheorem{corollary}[theorem]{Corollary}
\newtheorem{prob}[theorem]{Problem}

\theoremstyle{thmstyletwo}%

\newtheorem{remark}[theorem]{Remark}

\theoremstyle{thmstylethree}%

\numberwithin{equation}{section}

\allowdisplaybreaks
\begin{document}

%====================================================================
\title[A Space-Time Variational Method for Optimal Control]{A Space-Time Variational Method for Optimal Control Problems: Well-posedness, stability and numerical solution}

\author{\fnm{Nina} \sur{Beranek}}\email{nina.beranek@uni-ulm.de}
\author{\fnm{M.\ Alexander} \sur{Reinhold}}\email{martin.reinhold@uni-ulm.de}
\author{\fnm{Karsten} \sur{Urban}}\email{karsten.urban@uni-ulm.de}
\affil{\orgname{Ulm University}, 
	\orgdiv{Institute for Numerical Mathematics}, 
	\orgaddress{\street{Helmholtzstr.\ 20}, 
	\postcode{89081} \city{Ulm}, \country{Germany}}}

%-------------------------------------------------------------------------------------
\abstract{
We consider an optimal control problem constrained by a parabolic partial differential equation (PDE) with Robin boundary conditions. We use a well-posed space-time variational formulation in Lebesgue--Bochner spaces with minimal regularity. The abstract formulation of the optimal control problem yields the Lagrange function and Karush--Kuhn--Tucker (KKT) conditions in a natural manner. This results in space-time variational formulations of the adjoint and gradient equation in Lebesgue--Bochner spaces with minimal regularity. Necessary and sufficient optimality conditions are formulated and the optimality system is shown to be well-posed.

Next, we introduce a conforming uniformly stable simultaneous space-time (tensorproduct) discretization of the optimality system in these Lebesgue--Boch\-ner spaces. Using finite elements of appropriate orders in space and time for trial and test spaces, this setting is known to be equivalent to a Crank--Nicolson time-stepping scheme for parabolic problems. Differences to existing methods are detailed.

We show numerical comparisons with time-stepping methods. The space-time method shows good stability properties and requires fewer degrees of freedom in time to reach the same accuracy.
}
%-------------------------------------------------------------------------------------



%-------------------------------------------------------------------------------------

\maketitle

%====================================================================

%====================================================================
\section{Introduction}\label{Sec:1}
%====================================================================
The optimal control of partial differential equations (PDE) is an area of vast growing significance e.g.\ in fluid flows, crystal growths or medicine, see, e.g.\ \cite{f.troeltzsch2009,MR2516528}. This explains the huge amount of literature concerning theoretical as well as numerical aspects. 

The abstract form of such problems relies on a cost function $J:\cY\times\cU\to\R$, where $\cY$ and $\cU$ are function spaces for the \emph{state} $y$ and the \emph{control} $u$. The constrained optimal control problem then takes the form
 \begin{align}\label{eq:OCP}
 	J(y,u)\to\min!\quad
	\text{s.t.  } e(y,u) = 0,
 \end{align}
where the constraint $e(y,u) = 0$ is often termed as \emph{state equation}. At this point, there is a bifurcation concerning the subsequent approach. On the one hand, the \emph{first-discretize-then-optimize} approach seeks for an appropriate discretization of \eqref{eq:OCP} and then derives optimality conditions for the discretized optimal control problem. On the other hand, \emph{first-optimize-then-discretize} means that optimality conditions are derived directly w.r.t.\ \eqref{eq:OCP} and then the arising optimality system is discretized. We shall follow the second approach. 

\paragraph{First-optimize-then-discretize} 
Within this approach, the first step is a suitable interpretation of the state equation. In case of a PDE-constrained optimal control problem, the state equation is a PDE. Here, we are interested in the case where the PDE is a parabolic problem in space and time. This offers a variety of different formulations of the state equation, e.g.
\begin{compactitem}
\item \emph{Strong form}: $e(y,u) = 0$ is interpreted pointwise. This, however does often not allow statements on the well-posedness of the state equation. 
\item \emph{Semi-variational}: Using a method of lines yields either an inital value problem of an ordinary differential equation or a system of elliptic boundary value problems. 
\item \emph{Space-time variational}: Space and time are both treated as variables in a variational sense. In that case, the state equation is tested by space-time test functions $z\in\cZ$, where $\cZ$ is an appropriate Lebesgue--Bochner space, and takes the form, for a right-hand side $f(\cdot;u)\in \cZ'$
\begin{align}\label{eq:Prob}
	\text{find } y\in \cY:\quad b(y,z) = f(z; u)\quad \text{ for all } z\in \cZ.
\end{align}
\end{compactitem}

\paragraph{Space-time variational formulations and adjoint problem}  
We follow the last-mentioned method in the above list. In the literature, this approach has already been studied, see e.g.\ \cite{MR2407012,MR2874969,MR2861431,MR3343358,MR4076464,MR4223221}, but with some (partly subtle) differences to our approach to be detailed below. The well-posedness theory for  \eqref{eq:Prob} dates back (at least) to the 1970s, see e.g.\ \cite{j.l.lions1971,lions.magenes.2,r.dautrayj.lions1992}. In order to describe to which extent our approach differs from the mentioned papers, we need to detail the choices of the bilinear form $b(\cdot,\cdot)$ as well as the trial and test spaces $\cY$ and $\cZ$. We shall see that our subsequent choice allows to show well-posedness of \eqref{eq:Prob} under minimal regularity assumptions, but this necessarily implies that $\cY\not=\cZ$ and these spaces need to satisfy an inf-sup condition, which is known to hold, \cite{c.schwabr.stevenson2009,k.urbana.t.patera2012}. We postpone a detailed comparison in \S\ref{sec:spacetime} (in particular Remark \ref{Rem:Diff1}) below. 

Another difference of the known approaches from the literature and our proposed framework is the derivation of an optimality system. For the details, we refer to \S\ref{Sec:3.3}, Remark \ref{Rem:Diff2} below. We use the variational form \eqref{eq:Prob} to derive the reduced problem (w.r.t.\ the control), which allows us to prove the existence of a unique optimal solution. In a next step, we define the Lagrange function, again using the variational form \eqref{eq:Prob}, from which we then derive the Karush--Kuhn--Tucker (KKT) conditions. The adjoint problem arises in a natural variational form by the KKT system, see Proposition \ref{propos:KKTModelProblem_general} below.

\paragraph{Space-time discretization}  
In a final step, we propose a conforming space-time discretization, which amounts to construct finite-dimensional spaces $\mathcal{Y}_\delta \subset \mathcal{Y}$ and $\mathcal{Z}_\delta \subset \mathcal{Z}$ for a Petrov--Galerkin discretization of \eqref{eq:Prob} and later also the control space $\mathcal{U}_\delta \subset \mathcal{U}$. Since $\cY\not=\cZ$, the discrete spaces $\mathcal{Y}_\delta$ and  $\mathcal{Z}_\delta$ need to satisfy a discrete inf-sup condition, also known as \emph{Ladyshenskaja--Babu\v{s}ka--Brezzi (LBB)} condition, i.e., 
\begin{align}\label{eq:LBB}
	\inf_{y_\delta \in \cY_\delta} \sup_{z_\delta \in \cZ_\delta} 
	\frac{b(y_\delta, z_\delta)}{\| y_\delta\|_\cY\, \| z_\delta\|_\cZ}
	\ge \beta >0
\end{align}
uniformly in $\delta$ (where $\beta$ is independent of $\delta$). The inf-sup constant $\beta$ is particularly relevant as the Xu--Zikatanov lemma \cite{MR1971217} yields an error/residual-relation with the multiplicative factor $\frac1\beta$. In some cases, one can realize \emph{optimally stable} discretizations, i.e., $\beta=1$ (in particular, the constant is independent of the final time, which is crucial for optimal control problems), \cite{r.andreev2012,k.urbana.t.patera2012,k.urbana.t.patera2014}. This is a key motivation for our approach. 
However, there are  different stable discretizations described in the literature. For example, in \cite{GunzburgerKunoth2011,c.schwabr.stevenson2009,stevenson2021waveletintime} wavelet methods have been used to derive an LBB-stable discretization, \cite{r.andreev2012,k.urbana.t.patera2012,k.urbana.t.patera2014} propose tensorproduct discretizations (some of them reducing to time-stepping schemes) and \cite{MR4223221,o.steinbach2015} introduce unstructured finite element discretizations in space and time. Here, we use a tensorproduct discretization since they allow for efficient numerical solvers and admit optimal stability, \cite{j.henning.etal2019,palitta2019matrix}; of course, also other schemes could be used instead. Our approach leads to a different discrete system as previous approaches, see \S\ref{Sec:discoptsys}, Remark \ref{Rem:Diff3} below.

Until recent it has been believed that a simultaneous discretization of time and space variables would be way too costly since problems in $n+1$ dimension need to be solved, where $n$ denotes the space dimension. This has changed somehow since it is nowadays known that space-time discretizations yield good stability properties, can efficiently be used for model reduction and can also be treated by efficient numerical solvers, see \cite{TDCM14,DemGop11,c.schwabr.stevenson2009,r.andreev2016A,r.andreev2012,Yano14,c.mollet2016,j.henning.etal2019,palitta2019matrix}, just to name a few papers in that direction. However, the issues of a suitable discretization and the question if the arising higher-dimensional problem can efficiently be solved remain. Of course, also for space-time approaches different from ours, there are also efficient numerical solvers known, see e.g.\ \cite{MR3328986,MR2872584,MR2516528}.

\paragraph{Model problem} 
We consider the following PDE-constrained optimal control problem. 

\begin{prob}[Model problem in classical form]\label{prob:classicalFormulationOptContr} 
Let $I=(0,T)\subset \mathbb{R}$, $0<T<\infty$ and $\Omega \subset \mathbb{R}^n$ be a bounded open Lipschitz domain. The normal vector of $\partial \Omega \coloneqq \Gamma $ is denoted by $\nu (x) \in \mathbb{R}^n$ for all $x\in \Gamma$. 

The state space $\cY$ consists of mappings $y:I\times \Omega \to \mathbb{R}$, the control space $\cU$ of functions $u: I \times \Omega \to \mathbb{R}$. We are interested in determining a control $\overline{u}\in\cU$ and a corresponding state $\overline{y} \in\cY$ that solve the following optimization problem:
\begin{align*}
%	u^* = \mathrm{arg} 
	\min\limits_{ \left(y,u\right) \in\cY\times\cU } J(y,u) \coloneqq 
	\ts{\frac{1}{2}}    \int\limits_{\Omega} \lvert y(T,x) - y_d(x) \rvert^2 \diff x  + 
	\ts{\frac{\lambda}{2}}  \int\limits_{I} \int \limits_{\Omega} \lvert u(t,x) \rvert^2 \diff x \diff t 
\end{align*}
s.t. \vspace*{-8pt}
\begin{empheq}[box=\fbox]{alignat=2}\label{eq:pde}
	 {\dot y}(t,x)- \Laplace y(t,x) 
	 	& = {R}u(t,x) 
	 	&&\qquad \quad \text{in } I \times \Omega, \notag\\
	  \partial_{\nu} y(t,x)+  \mu(x) \cdot y(t,x)
	  	&=\eta(t,x) 
	  	&&\qquad \quad \text{in } I \times \Gamma,  \\ 
	  y(0,x) 
	  	& = {0} 
		&& \qquad \quad \text{in } \Omega, \notag  
\end{empheq}
where the functions $\mu: \Omega \to \mathbb{R}$, $\eta :I \times \Gamma \to \mathbb{R}$ {and $y_d: \Omega \to \mathbb{R}$} as well as a scalar $\lambda > 0$ are given. {Moreover, $R$ is a linear operator, whose role will be described below.} 
We shall always assume that $\mu(x)>0$\footnote{{We note that we do not need strict positivity in order to ensure well-posedness. But it allows us to use energy norms in the sequel.}} for all $x \in \Omega$ a.e..
\end{prob}

\begin{remark}\label{Rem:1}(a) 
	We could easily extend to a cost function of the form 
	\begin{align*}
		J(y,u) &= \ts{\frac{\omega_1}{2}} \| y-y_d\|_{L_2(I; L_2(\Omega))}^2 
			+ \ts{\frac{\omega_2}{2}} \| y(T)-y_d(T)\|_{L_2(\Omega)}^2 
			+ \ts{\frac{\omega_3}{2}}   \| u\|_{L_2(I; L_2(\Omega))}^2, 
	\end{align*}
	with real constants $\omega_1$, $\omega_2\geq 0$, $\omega_1+\omega_2>0$, $\omega_3>0$  
	and $y_d:I\times\Omega\to\R$. \\
	(b) The extension to inhomogeneous initial conditions $y(0,x)=y_0$ and other types of boundary conditions follows standard lines, e.g.\ \cite{c.schwabr.stevenson2009} and Remark \ref{Rem:ICs}. \\
	(c) In the first preprint version of this paper, we considered box constraints for the control. In order to discuss the analysis concerning well-posedness and convergence in full detail, we decided to devote control constraints to future research.
\end{remark}

\paragraph{Organization of the paper} 
The remainder of this paper is organized as follows. In Section \ref{Sec:2}, we recall and collect some preliminaries on {PDE-}constrained optimization problems in reflexive Banach spaces and on space-time variational formulations of parabolic PDEs. The space-time variational formulation of the optimal control problem under consideration is developed in Section \ref{Sec:3}. In particular, we derive necessary and sufficient optimality conditions. Section \ref{Sec:4} is devoted to the space-time discretization of the PDE, the discretization of the control as well as of the adjoint problem. The latter one turns out to be much simpler in our space-time context than in the semi-discrete setting as we obtain a linear system whose matrix is just the transposed of the matrix appearing in the primal problem. The fully discretized optimal control problem is then solved {numerically}. We report on our numerical experiments in Section \ref{Sec:5} and conclude by a summary, conclusions and an outlook in Section \ref{Sec:6}.

%====================================================================
\section{Preliminaries}\label{Sec:2}
%====================================================================
Let us start by collecting some preliminaries that we will need in the sequel.

%------------------------------------------------------------------------------------------------------------------------
\subsection{Optimal control problems}\label{SubSec:2.1}
%------------------------------------------------------------------------------------------------------------------------
In this section, we recall the abstract functional analytic framework for optimal control problems in reflexive Banach spaces\footnote{Concerning the chosen model problem we will deal with real Hilbert spaces. However, we will not identify these Hilbert spaces with their dual spaces, which is the reason why we describe the general optimal control framework for reflexive Banach spaces.} which we will later apply within the space-time setting. The consideration of control and/or state constraints is devoted to future research.

\begin{prob}\label{prob:OptAllg}
Let $\cY$, $\cU$, $\cZ$ be some real {reflexive} Banach spaces. Given an \emph{objective function} $J:\cY \times \cU \to \mathbb{R}$ and the \emph{state operator} $e:\cY\times \cU \to \cZ'$, we consider the constrained optimization problem 
\begin{align*}
	\min\limits_{ (y,u) \in \cY\times \cU } J(y,u)
	\quad
	\text{subject to (s.t.) the {constraint} } e(y,u)=0.
\end{align*}
\end{prob}

\begin{remark}\label{rem:constraint}
Note, that $e(y,u)=0$ is an equation in the dual space $\cZ'$ of $\cZ$. {Since we consider reflexive Banach spaces, it holds $\cZ''\cong\cZ$. Therefore, the constraint is to be interpreted as} 
%%which means that the constraint is to be interpreted as
%\begin{align}\label{eq:constraint}
	$\langle e(y,u), z\rangle_{\cZ'\times{\cZ}} = 0$ 
	for all $z\in{\cZ}$, 
%\end{align}
where $\langle \cdot, \cdot\rangle_{\cZ'\times{\cZ}}$ is the duality pairing, {and the adjoint state will be in $\cZ$}.
%This is also the reason why we use $\cZ'$ as opposed to the standard notation -- i.e.,  {the adjoint state will be in $\cZ''$, which is $\cZ$, if $\cZ$ is reflexive, which is often the case.}
\end{remark}

A pair {$( \overline{y},\overline{u})\in\cY \times \cU$} is called \emph{local optimum} of Problem \ref{prob:OptAllg} if 
\begin{align}\label{eq:lokOptimality}
	J(\overline{y},\overline{u}) \leq J(y,u) 
	\qquad 
	\forall (y,u)\in \mathcal{N}(\overline{y},\overline{u}) \cap ({\cY \times \cU}) \text{,}
\end{align}
for some neighborhood $\mathcal{N}(\overline{y},\overline{u})$ of $(\overline{y},\overline{u})$; the pair is called \emph{global optimum} of Problem \ref{prob:OptAllg} if  \eqref{eq:lokOptimality} is satisfied for all $(y,u)\in {\cY \times \cU}$. 

We will be investigating the well-posedness of such optimal control problems in a space-time variational setting. This requires first to study the well-posedness of the \emph{state equation} $e(y,u)=0$, namely the question if a unique state can be assigned to each admissible control. If so, one defines the \emph{control-to-state operator}
\begin{align}\label{eq:ControlToStateOperator}
	S: \cU \to \cY\text{, } u \mapsto y(u) = Su,
\end{align}
which allows one to consider the \emph{reduced objective function} 
%\begin{align}\label{eq:reduzierteZielfunktion}
	$\hat{J}: \cU \to \mathbb{R}\text{, }\hat{J}(u) \coloneqq J(Su,u)$ 
%\end{align}
and the corresponding \emph{reduced problem}
\begin{align}\label{eq:reducedProblem}
	{\min\limits_{u\in \cU} \hat{J}(u)}.
\end{align}
We recall the following well-known result for later reference, \cite[Thm.\ 5.1]{j.delosreyes2015}, \cite[Thm.\ 1.46]{m.hinzer.pinnaum.ulbrichs.ulbrich2009}.

\begin{theorem}\label{satz:Existenz1} 
Let {$\cU \neq \emptyset$} and  $\hat{J}:\mathcal{U} \to \mathbb{R}$ be a weakly lower semi-continuous function such that there exists a constant $c>-\infty$ so that $\hat{J}(u)\geq c$  for all $u\in{\cU}$. Then, \eqref{eq:reducedProblem} has at least one solution $\overline{u}$. If $\hat{J}$ is in addition strictly convex, then the optimal solution is unique.
\end{theorem}

Necessary {first order} optimality conditions for optimal control problems are based upon {the Euler--Lagrange equation $\hat{J}^\prime (\overline{u}) = 0$}, e.g.\ \cite{f.troeltzsch2009}. This, however,  involves the derivative of $\hat{J}$, which is often difficult to determine exactly. The well-known way-out is through the adjoint problem. In fact, if $e_y(Su,u): \cY \to \cZ'$ (the partial derivative of $e(\cdot,\cdot)$ w.r.t.\ $y$) is a bijection, then, 
%\begin{align}
%	\label{eq:reduzierteZielfunktion2}
	$\hat{J}'(u) = J_u(Su,u)-e_u(Su,u)^*\left(e_y(Su,u)^*\right)^{-1}J_y(Su,u)$,  
%\end{align}
for any $u\in \cU$, where $e_y(Su,u)^*$ and $e_u(Su,u)^*$ denote the adjoint operators of $e_y(Su,u)$ and $e_u(Su,u)$, respectively. In order to avoid the determination of the inverse of the adjoint $e_y(Su,u)^*$, one considers the \emph{adjoint equation}
\begin{align}\label{eq:adjointEquation}
	e_y(y,u)^*{z}={-}J_y(y,u),
\end{align}
whose solution {${z} \in \cZ$} is called \emph{adjoint state}. Then, 
\begin{align}\label{eq:reduzierteZielfunktionAdjoint}
	\hat{J}^\prime (u)%= J_u(y(u),u) {+} e_u(y(u),u)^*{z} 
	= J_u(Su,u)+e_u(Su,u)^*{z}.
\end{align}

\begin{theorem}[KKT system] \label{theo:KKTBoxConstraint} 
Let $\overline{u}$ be a solution of  \eqref{eq:reducedProblem} and $\overline{y}\coloneqq S\overline{u}$ the related state. Then, there exists an adjoint state {$\overline{{z}}\in \cZ$}, 
such that the following KKT system is satisfied for all $(t,x) \in I\times \Omega$ a.e.:%
\begin{subequations}\label{eq:kom}
	\begin{align}
		e(\overline{y},\overline{u})&=0, 
			&& \label{eq:komp4}\\
		e_y(\overline{y},\overline{u})^*\overline{{z}}
		&={-}J_y(\overline{y},\overline{u}),  
			&& \label{eq:komp5}\\
		e_u(\overline{y},\overline{u})^*\overline{{z}} 
		&= {-}J_u(\overline{y},\overline{u}). 
			&& \label{eq:cond_J_u} 
	\end{align}
\end{subequations}
\end{theorem}

The Lagrange function $\mathcal{L}: \cY \times \cU \times {\cZ} \to \mathbb{R}$ to Problem \ref{prob:OptAllg} reads
\begin{align*}
	\mathcal{L}(y,u,{z})
		\coloneqq  J(y,u) {+} \langle {z}, e(y,u) \rangle_{{\cZ} \times \cZ'}
\end{align*}
Then, (\ref{eq:kom}) can equivalently be written as $\nabla\mathcal{L}(\overline{y},\overline{{z}},\overline{{z}})=0$.

%------------------------------------------------------------------------------------------------------------------------
\subsection{Space-time variational formulation of parabolic problems}\label{sec:spacetime}
%------------------------------------------------------------------------------------------------------------------------
In order to detail the setting in \S\ref{SubSec:2.1} for the specific Problem \ref{prob:OptAllg} at hand, we review a variational formulation of the initial boundary value problem \eqref{eq:pde} in space and time, which yields the specific form of the state operator $e(\cdot,\cdot)$. To this end, {let $H\coloneqq L_2(\Omega)$, $G\coloneqq L_2(\Gamma)$, $V\coloneqq H^1(\Omega)$ and $V'$ be the dual of $V$ induced by the $H$-inner product. Then, we denote the Lebesgue--Bochner spaces by $\cH\coloneqq L_2(I;H)$, $\cG\coloneqq L_2(I;G)$, $\cV\coloneqq L_2(I;V)$ and $\cV'\coloneqq L_2(I;V')$. Moreover, denoting by $\langle\cdot,\cdot\rangle_{V'\times V}$ the duality pairing in space only, we obtain inner products and duality pairing in time and space as
\begin{align*}
	(u,v)_\cX\coloneqq  \int_I (u(t),v(t))_X\dt,
	\qquad
	\langle u,v\rangle_{\cV'\times\cV} \coloneqq  \int_I \langle u(t), v(t)\rangle_{V'\times V} \dt
\end{align*}
for the respective $u$ and $v$ and $X\in\{ V',H,V\}$, $\cX\in\{\cV',\cH,\cV\}$, respectively.
} 

{Then, we} start by testing the first equation in \eqref{eq:pde} with functions ${z}(t)\in {V}$, $t\in I$ a.e., integrate over time, perform integration by parts in space and insert the Robin boundary condition of \eqref{eq:pde}. Denoting by $a: {V}\times{V}\to \mathbb{R}$ the bilinear form in space, i.e., {$a(\phi,\psi) \coloneqq (\nabla \phi, \nabla \psi)_H + (\mu\,\phi, \psi)_G$}, we get 
\begin{equation} \label{eq:ARWPIntegriertSpaceMitBilA}
		\langle {\dot y}(t), {z}(t) \rangle_{{V'\times V}}
			+  a(y(t),{z}(t))  
			= {\langle Ru(t),z(t)\rangle_{V'\times V} +  (\eta(t), z(t))_G},  
\end{equation}
for $t\in I$ a.e. To obtain a variational formulation in space and time we integrate \eqref{eq:ARWPIntegriertSpaceMitBilA} in time and obtain
\begin{equation} 
	\label{eq:varformSTprev}
	{ 
	\langle \dot{y}, z\rangle_{\cV'\times\cV} }
	+ \!\int\limits_{I}\!\! a(y(t),{z}(t)) \diff t 
	= {\langle Ru,z\rangle_{\cV'\times \cV} +  (\eta, z)_\cG .}
\end{equation}
The trial space for the {the state} $y$ is a Lebesgue--Bochner space defined as 
\begin{align}\label{eq:cY}
	\mathcal{Y} \coloneqq \lbrace y \in L_2(I;{V}): {\dot{y}} \in  L_2(I;{V}'), {y(0)=0} \rbrace 
		= L_2(I;{V})\cap  {H^1_{(0)}(I;{V}'}).
\end{align}
{As in \cite{k.urbana.t.patera2012,k.urbana.t.patera2014}, we choose the norms 
\begin{align*}
	\|y\|_\cY^2\coloneqq  \| \dot{y}\|_{\cV'}^2 + \| y\|_\cV^2+\| y(T)\|_H^2 
	\quad\text{and}\quad 
	\|\phi\|_V^2:=a(\phi,\phi), 
\end{align*}
but other equivalent norms can also be considered. The test space reads}
\begin{align}\label{eq:cZ}
	\mathcal{Z}\coloneqq  {\cV=L_2(I;V),
	\qquad
	\| \cdot\|_\cZ\coloneqq  \| \cdot\|_\cV.
	}
\end{align}
For the well-posedness of \eqref{eq:varformSTprev} (see \cite{c.schwabr.stevenson2009}), we need $Ru\in\cV'=L_2(I; V')=\cZ'$. However, the definition of the cost function $J$ in Problem \ref{prob:classicalFormulationOptContr} requires
\begin{align}\label{eq:cU}
	\cU \coloneqq  {\cH=L_2(I; H)},
	\qquad
	\| \cdot\|_\cU\coloneqq  \| \cdot\|_{\cH}.
\end{align}%
Thus, we can now detail the role of the linear mapping $R$, namely $R:\cH\to\cV'$ (which could here also be just the identity). 
We introduce the bilinear form $b: \cY\times \cZ \to\R$ and the linear form $h\in\cZ'$ by
\begin{align*}
	b(y,z) \coloneqq  \langle \dot y, z\rangle_{\cV'\times\cV}
	+ \int\limits_{I} a(y(t),{z}(t)) \diff t,  
%	\,\, f(u,z) \coloneqq  \langle u,z\rangle_{\cV'\times\cV}, 
	\quad
	\,\, h(z)\coloneqq (\eta,z)_\cG,
\end{align*}
so that \eqref{eq:varformSTprev} equivalently can be written as
\begin{align}\label{eq:varformST}
	b(y,z)=\langle Ru,z\rangle_{\cV'\times\cV} + h(z) \qquad \forall z\in\cZ.
\end{align}
Obviously, \eqref{eq:varformST} is a variational problem {of the form \eqref{eq:Prob}, where the right-hand side is a linear form in $\cZ'$ for all $u\in\cU$. 
For later reference, it will be convenient to reformulate \eqref{eq:varformST} in operator form. To this end, we define
\begin{align}
	B: \cY\to\cZ',	&\quad \langle By, z\rangle_{\cZ'\times\cZ} \coloneqq  b(y,z),  
%		\nonumber\\
%	{M}: \cU\to{\cU'},	
%			&\quad \langle {M}u, \delta u\rangle_{{\cU'\times\cU}} 
%			\coloneqq \langle u,\delta u\rangle_{\cV'\times\cV}, 
	\label{eq:DefOp}
\end{align}
so that \eqref{eq:varformST} reads $By=Ru+h$. If we define the differential operator in space as $A_x:V\to V'$ by $\langle A_x\phi,\psi\rangle_{V'\times V}:=a(\phi,\psi)$ with its space-time extension $A:\cV\to\cV'$ defined as $\langle A y, \delta y\rangle_{\cV'\times\cV}:=\int_I a(y(t),\delta y(t))\, dt$, then we get the representation $By=\dot{y}+A y$.}

\subsubsection*{Well-posedness of the parabolic problem}
The proof of the well-posedness of the variational form \eqref{eq:varformST} for any given $u\in\cU$ basically follows the lines of \cite{r.dautrayj.lions1992,c.schwabr.stevenson2009,k.urbana.t.patera2012}, namely by verifying the conditions of the Banach--Ne\v{c}as theorem.
For the Robin data we make the usual assumptions $\mu\in L_\infty(I\times\Gamma)$ and $\eta\in \cG=L_2(I;G)$. The surjectivity is shown by proving the convergence of a Faedo--Galerkin approximation, \cite[App.\ A]{c.schwabr.stevenson2009}. Inf-sup-condition and boundedness can be derived by detailing primal and dual supremizers. 

\begin{proposition}\label{Prop:WellPosed}
	The problem \eqref{eq:varformST} is well posed with
	\begin{align}
		\inf_{y\in\cY} \sup_{z\in\cZ} \frac{b(y,z)}{\| y\|_\cY\, \| z\|_\cZ}
		&= \inf_{z\in\cZ}  \sup_{y\in\cY} \frac{b(y,z)}{\| y\|_\cY\, \| z\|_\cZ}
			\label{eq:infsup} \\
		&= \sup_{y\in\cY} \sup_{z\in\cZ} \frac{b(y,z)}{\| y\|_\cY\, \| z\|_\cZ}
		= \sup_{z\in\cZ}  \sup_{y\in\cY} \frac{b(y,z)}{\| y\|_\cY\, \| z\|_\cZ}
		= 1, \nonumber
	\end{align}
	in particular $\| B\|_{\cY\to\cZ'} = \| B^*\|_{\cZ\to\cY'} = \| B^{-1}\|_{\cZ'\to\cY} = \| B^{-*}\|_{\cY'\to\cZ} = 1$.
\end{proposition}
\begin{proof}
	The proof closely follows the lines of \cite[Thm.\ 5.1]{c.schwabr.stevenson2009}, \cite[Prop.\ 1]{k.urbana.t.patera2012} and \cite[Prop.\ 2.6]{k.urbana.t.patera2014}. In fact, we can identify primal and dual supremizers for given $z\in\cZ$ and $y\in\cY$, respectively, as follows
	\begin{align*}
		\cZ\ni s_y
		&:= \arg \sup_{\delta z\in\cZ} \frac{b(y,\delta z)}{\| \delta z\|_\cZ}
		= A^{-1} B y 
		= A^{-1}\dot{y} + y, \\
		\cY\ni s_z
		&:= \arg \sup_{\delta y\in\cY} \frac{b(\delta y,z)}{\| \delta y\|_\cY}
		= B^{-1} A z. 
	\end{align*}
	In addition, $\| s_y\|_{\cZ}^2 = \| A^{-1}\dot y + y\|_\cV^2 = \| \dot y\|_{\cV'}^2 + \| y\|_\cV^2 + \| y(T)\|_H^2 =\| y\|_\cY^2$ and $\| s_z\|_{\cY} = \| \dot{s_z} + As_z\|_{\cV'} = \| Az\|_{\cV'}= \| z\|_\cZ$, which completes the proof.
\end{proof}

\begin{remark}
	Even though we have proven optimal stability and continuity, we will later also need the general case, in which we have
	\begin{align}
		\| B\|_{\cY\to\cZ'} = \| B^*\|_{\cZ\to\cY'} =: \gamma_B,\quad
		& \| B^{-1}\|_{\cZ'\to\cY} = \| B^{-*}\|_{\cY'\to\cZ} = \ts{\frac1\beta}.
	\end{align}
\end{remark}

\begin{remark}[Inhomogeneous initial conditions]\label{Rem:ICs}
	As already mentioned in Remark \ref{Rem:1}, we restrict ourselves to homogeneous initial conditions only for convenience of the presentation. In fact, for $y_0\not=0$, we would set $\cY:=L_2(I;V)\cap H^1(I;V')$, the test space would be $\cZ:= L_2(I;V)\times H$ and bilinear and linear forms read for $y\in\cY$, $z=(z_1,z_2)\in\cZ$
	\begin{align*}
		b(y, (z_1,z_2))&\coloneqq 
			\langle \dot y, z_1\rangle_{\cV'\times\cV} + \int\limits_{I} a(y(t),z_1(t)) \diff t + (y(0),z_2)_H,\\
		f((z_1,z_2);u) &\coloneqq
			\langle Ru,z_1\rangle_{\cV'\times\cV} + h(z_1) + (y_0,z_2)_H,
	\end{align*}
yielding a state equation of the form \eqref{eq:Prob}. Hence, inhomogeneous initial conditions can be treated analogously, just the notation becomes a bit more heavy, \cite{c.schwabr.stevenson2009}.
\end{remark}

\subsubsection*{Comparison with other space-time methods}
As already mentioned in the introduction, our approach is somehow different as existing ones in the literature. We are now going to describe the differences concerning the formulation of the state equation in more detail.

\begin{remark}[Differences to existing space-time methods]\label{Rem:Diff1}\ \hfill
\begin{compactenum}[(a)]
\item In \cite{MR2407012,MR2874969,MR2861431}, Meidner, Neitzel and Vexler use (almost) the same trial space $\cY$ as in \eqref{eq:cY}, namely $\widetilde\cY:=L_2(I;V)\cap H^1(I;V')$, but impose the initial condition $y(0)=y_0$\footnote{Recall, that we have chosen homogeneous initial conditions $y_0=0$ only for simplicity of exposition, see Remark \ref{Rem:ICs}.} in strong form. The arising problem is not of the form \eqref{eq:Prob}. In fact, the well-posedness does not follow from the Banach--Ne\v{c}as theorem but with techniques from semigroup theory, \cite{r.dautrayj.lions1992}. This requires $y_0\in V$ ($y_0\in H$ is the minimal requirement) and additional regularity, \cite[Prop.\ 2.1]{MR2407012}. In fact, the right-hand side is required to be in $L_2(I;H)$ and the solution is in $L_2(I;V\cap H^2(\Omega))\cap H^1(I;H)\hookrightarrow C(\bar{I};V)$, which is significantly stronger than $\cY$ defined in \eqref{eq:cY}.\\
Moreover, treating the initial condition as in \cite{MR2407012,MR2874969,MR2861431} allows to use $\widetilde\cY$ also as test space, which is another reason for the additional smoothness, but which yields a Galerkin discretization instead of a Petrov--Galerkin one, see below. 
%-----------
\item In \cite{MR3343358,MR4076464}, von Daniels, Hinze and Vierling use the same trial space $\widetilde\cY$ for the state equation as \cite{MR2407012,MR2874969,MR2861431}, but impose the initial condition in a weak sense, \cite[(1.5)]{MR3343358}. Moreover, $\widetilde\cY$ is chosen also as test space (in a Galerkin spirit).
%-----------
\item In the more recent paper, Langer, Steinbach, Tr\"{o}ltzsch and Yang use the same variational formulation as we do, \cite{MR4223221}. However, the initial condition is part of the definition of the trial space, which will be relevant for the adjoint problem.
\end{compactenum}
\noindent
In all cases, there are differences to our approach in the derivation/formulation of the adjoint equation and the adjoint state to be described in the next section.
\end{remark}

%====================================================================
\section{Space-Time Variational Optimal Control Problem}\label{Sec:3}
%====================================================================

Next, we formulate the optimal control problem in the variational space-time setting by specifying the above abstract framework. 
%%------------------------------------------------------------------------------------------------------------------------
%\subsection{Formulation}
%%------------------------------------------------------------------------------------------------------------------------
To do so, we are now going to derive a space-time variational formulation of Problem \ref{prob:classicalFormulationOptContr}. The state space is determined by the PDE, i.e., we choose $\cY$ defined in \eqref{eq:cY}. %Since we do not consider state constraints in Problem \ref{prob:classicalFormulationOptContr}, we choose $\cY_{ad}= \cY$.
In view of Remark \ref{rem:constraint} and recalling that $\cZ'' \cong \cZ$, we are now in position to formulate the constraint in space-time variational form as follows
\begin{align}\label{eq:constraint-spacetime}
	\langle e(y,u), z\rangle_{\cZ'\times\cZ} \coloneqq   
		b(y,z) - \langle {R}u,z\rangle_{\cV'\times\cV}  - h(z),
	\qquad z\in\cZ,
\end{align}
i.e., $e(y,u)\coloneqq  By-Ru-h \in\cZ'$. Next, we can detail the control-to-state operator $S: \cU\to\cY$ as follows $Su = B^{-1} \left( {R}u+{h} \right)$. 
Finally, the objective function $J: \cY \times \cU \to \mathbb{R}$ in Problem \ref{prob:classicalFormulationOptContr} can now be written as
$J(y,u) = \frac{1}{2}  \lVert y(T) - y_d \rVert_{{H}}^2 + \frac{\lambda}{2}  \lVert u \rVert_{{\cH}}^2$,  where $\lambda > 0$ is the regularization parameter.

%------------------------------------------------------------------------------------------------------------------------
\subsection{Existence of an optimal solution}
%------------------------------------------------------------------------------------------------------------------------

\begin{prob}[Reduced problem]\label{prob:reducedProblemModelProblem} 
Find {a} control ${\bar{u}\in \cU}$ such that
\begin{align}\label{eq:hatJ}
	{\bar{u} = \arg}\min\limits_{u\in \cU} \hat{J}(u), 
	&\quad
	 \hat{J}(u)\coloneqq  \ts{\frac{1}{2}}  \lVert (Su)(T) - y_d \rVert_{{H}}^2 
	 	+ \ts{\frac{\lambda}{2}}  \lVert u \rVert_{\cH}^2 . 
%	 &\text{s.t. } u \in\cU_{ad}.
\end{align}
\end{prob}
\begin{proposition}
	Problem \ref{prob:reducedProblemModelProblem} admits a unique solution.
\begin{proof}
	Since $\hat{J}$ is non-negative there exists a constant $c>-\infty$ so that $\hat{J}(u)\geq c$ for all $u\in\cU$. It remains to show that $\hat{J}$ is weakly lower semi-continuous. Since  $\hat{J}$ is easily seen to be strictly convex and continuous (by continuity of $S$ and the norms), Theorem \ref{satz:Existenz1} proves the claim.
\end{proof}
\end{proposition}
%------------------------------------------------------------------------------------------------------------------------
\subsection{First order necessary optimality conditions}\label{Sec:3.3}
%------------------------------------------------------------------------------------------------------------------------
Adopting the previous notation, we start by detailing the Lagrange function $\mathcal{L}: {\cY \times \cU \times \cZ} \to \mathbb{R}$ for Problem \ref{prob:classicalFormulationOptContr}, namely 
\begin{align}\label{def:LagrangeModelProblem_general}
	\mathcal{L}({y,u,z})
	&=\ts{\frac{1}{2}}  \lVert y(T) - y_d \rVert_{{H}}^2 
		+ \ts{\frac{\lambda}{2}}  \lVert u \rVert_{\cH}^2
 		+\langle {z}, By - {R}u - {h}\rangle_{{\cV \times \cV'}}  .
\end{align}
The partial derivatives can easily be derived as follows: $\mathcal{L}_z({y,u,z}) =   B{y}-{R}{u} -{h}$, $\mathcal{L}_y({y,u,z})=  D{y} - g  {+} B^*{z}${, where we introduce the bilinear form $d:\cY\times \cY\to\R$, $d(y,\delta y)\coloneqq (y(T),\delta y(T))_H$ and the associated operator\footnote{By the Cauchy-Schwarz inequality, we easily see that  $\|D\|_{\cY\to\cY'}\le 1$.} $D:\cY\to\cY'$, $ \langle Dy,\delta y\rangle_{\cY'\times\cY}\coloneqq d(y,\delta y)$ as well as the functional $g\in \cY'$, $g(\delta y) \coloneqq (y_d,\delta y(T))_H$. {Finally, for $\delta u\in\cU$, we have $\mathcal{L}_u(y,u,z)[\delta u] = \lambda\, (u,\delta u)_\cH - \langle z, R\delta u\rangle_{\cV\times\cV'}$. Hence, we obtain the following first order optimality (KKT) system: Find $\left( \overline{y}, \overline{z}, \overline{u} \right) \in \cY \times\cZ \times \cU$ such that}
{
\begin{subequations}\label{eq:KKT2}
  	\allowdisplaybreaks
		\begin{align}
			\label{eq:KKT2:a}
			b(\overline{y},\delta z) 
				&- \langle R\overline{u},\delta z\rangle_{\cV'\times\cV}
					&&\kern-38pt= h(\delta z) 
					& \forall \delta z\in\cZ, \\
			\label{eq:KKT2:b}
			d( \overline{y}, \delta y)
				&+b(\delta y,\overline{z}) 
					&&\kern-38pt= g(\delta y) 
					& \forall \delta y\in\cY, \\
			\label{eq:KKT2:c}
			\lambda ( \overline{u}, \delta u)_\cH 
				&- \langle R\delta u, \overline{z}\rangle_{\cV'\times\cV}
					&&\kern-38pt= 0
					& \forall \delta u \in \cU.
		\end{align}
\end{subequations}%
Let us further detail the \emph{gradient equation} \eqref{eq:KKT2:c}. Denote by $R^*:\cV\to\cH$ the adjoint operator of $R$ defined by $(R^*v,h)_\cH := \langle R h, v\rangle_{\cV'\times\cV}$ for $v\in\cV$ and $h\in\cH$. Then, \eqref{eq:KKT2:c} reads $\lambda ( \overline{u}, \delta u)_\cH - (R^*\overline{z},\delta u)_\cH=0$,  which means that we can derive a relation of the optimal control $\bar{u}$ and the optimal adjoint state $\bar{z}$, namely 
\begin{align}\label{Eq:uz}
	\bar{u} = \lambda^{-1}\, R^* \bar{z}.
\end{align}
Then, we can formulate the KKT conditions as follows.} 

\begin{proposition}[Optimality (KKT) system]\label{propos:KKTModelProblem_general} 
	Let $\left( \overline{y}, \overline{u} \right) \in \cY \times \cU$ be an optimal solution of Problem \ref{prob:classicalFormulationOptContr}.  {Then, there exists an adjoint state $\overline{z} \in\cZ$} such that the following optimality system holds:
	\begin{subequations}\label{eq:KKT1}
		\begin{align}
			B\overline{y} - R \overline{u}
					&= {h}		
					&& \text{ in } \cZ' 	
					& \text{\emph{(state equation)}}, \label{eq:KKT1:a} \\
			{D\overline{y} + B^*\overline{z}}
					&= {g }	
					&& \text{ in } \cY'	
					& \text{\emph{(adjoint equation)}},\label{eq:KKT1:b}\\
			 \lambda \overline{u} {-} {{R^*}  \overline{z}} 
					&= {0} && \text{ in } \cU'
					& \text{\emph{(gradient equation)}}, \label{eq:KKT1:c}
		\end{align}
	\end{subequations}
	{or, in operator form
	\begin{align*}
		\begin{pmatrix} 
			D & B^* & 0 \\
			B & 0 & -R \\
			0 & -R^* & \lambda I
		\end{pmatrix}
		\begin{pmatrix} \bar y \\ \bar z \\ \bar u \end{pmatrix}
		= \begin{pmatrix} g \\ h \\ 0 \end{pmatrix}.
	\end{align*}}
\end{proposition}

Setting $P:=RR^*:\cV\to\cV'$, inserting \eqref{Eq:uz} into \eqref{eq:KKT2} yields the reduced first order optimality system for determining $(\overline{y},\bar{z})\in\cY\times\cZ$ such that
\begin{subequations}\label{eq:KKT3}
\allowdisplaybreaks
		\begin{align}
			\label{eq:KKT3:a}
			b(\overline{y},\delta z) 
				&- \lambda^{-1} \langle P \bar{z},\delta z\rangle_{\cV'\times\cV}
					&&\kern-38pt= h(\delta z) 
					&& \forall \delta z\in\cZ, \\
			\label{eq:KKT3:b}
			d( \overline{y}, \delta y)
				&+b(\delta y,\overline{z}) 
					&&\kern-38pt= g(\delta y) 
					&& \forall \delta y\in\cY, 
			\end{align}
\end{subequations}%
or, in operator form
\begin{align}\label{eq:reducedoptimal}
	L \begin{pmatrix} \bar{y} \\ \bar{z} \end{pmatrix}
	= \begin{pmatrix} g \\ h \end{pmatrix},
	\quad
	L:= \begin{pmatrix} D & B^* \\ B & -\lambda^{-1} P \end{pmatrix}: 
	\cW:=(\cY\times\cZ)\to\cW'.
\end{align}
}%
From \eqref{eq:KKT2:b} we see that the adjoint problem arises from the primal one by exchanging the roles of trial and test spaces -- and by a different right-hand side, of course. 

%{\color{red}
%\begin{remark}
%{We note that our approach for the adjoint equation differs from some other approaches in the literature, e.g. \cite{MR3343358,MR4076464}, where also a space-time variational formulation is used.}
%%We note that our approach for the adjoint equation somewhat differs from \com{the approach e.g. used in \cite{MR3343358,MR4076464} where the authors consider} a space-time variational formulation for the adjoint equation. 
%{To our best understanding the authors there derive a \emph{strong} formulation of the adjoint equation first, and then apply a variational space-time formulation. This leads to a Galerkin setting, namely $\cZ = \cY$ which is not the case here, see \eqref{eq:cZ} and \eqref{eq:cY}.} 
%%This leads to a classical formulation of the adjoint problem yielding a
%%Galerkin setting\footnote{{Recall, that $\cY$ is the state space and $\cZ$ is the test space for the constraint equation.}}, namely $\cZ = \cY$ which is not the case here, see \eqref{eq:cZ} and \eqref{eq:cY}. 
%In particular, we do not derive a classical form of the adjoint equation, so that the adjoint state resides in ${\cZ}$.\footnote{{Recall, that $\cY$ is the state space and $\cZ$ is the test space for the constraint equation.}} To the very best of our knowledge, only \cite{GunzburgerKunoth2011} follows the same path.
%\end{remark}
%}

{
\begin{remark}[Well-posedness of the adjoint problem]
We recall from \S\ref{sec:spacetime} that the primal problem \eqref{eq:KKT2:a}  is well-posed, see e.g.\  \cite{r.dautrayj.lions1992,c.schwabr.stevenson2009}. This is a consequence of the Banach--Ne\v{c}as theorem and the fact that the inf-sup condition \eqref{eq:infsup} holds. Due to its specific form, this immediately implies well-posedness also of the adjoint problem \eqref{eq:KKT2:b}, even with the same inf-sup constant as for \eqref{eq:KKT2:a} .
\end{remark}
}

\begin{remark}
Due to the convexity of the objective function as well as the linearity of the state equation, the  Problem \ref{prob:classicalFormulationOptContr} is convex. Hence, every {solution of} \eqref{propos:KKTModelProblem_general} is {a global} optimal solution of the problem, \cite{f.troeltzsch2009}.
\end{remark}

\begin{theorem}[Well-posedness of the optimality system]\label{Thm:OptSystem}
	The first order optimality system \eqref{eq:reducedoptimal} is well-posed for all $\lambda>0$ with  
	\begin{align*}
		L^{-1} =
		\begin{pmatrix}
			\lambda^{-1} B^{-1} P C^{-1}
			& B^{-1} - \lambda^{-1} B^{-1} P C^{-1} D B^{-1} \\
			C^{-1} 
			& -C^{-1} D B^{-1}
		\end{pmatrix}
	\end{align*}
	with $C:= B^{*} + \lambda^{-1} DB^{-1} P: \cZ\to\cY'$. For $\gamma_P:=\| P\|_{\cV\to\cV'}$, we have %changed B^{-*} to B^{*} AR 
	\begin{align}\label{eq:estLinverse}
		\| L^{-1} \|_{\cW'\to\cW} \le
		\ts{\frac{\gamma_P+\lambda\beta + \lambda\beta^2}{\lambda\beta^3}}. 
	\end{align}
\end{theorem}
\begin{proof}
	The fact $LL^{-1}=L^{-1}L=I$ can be verified by straightforward calculations, as long as $L^{-1}$ exists. This, in turn, boils down to the existence of $C^{-1}$. In order to show this, note that $C=B^*(I+\lambda^{-1} B^{-*}DB^{-1} P) =: B^* (I+\lambda^{-1}K)$ and $B^*$ is invertible. The operator $K:\cZ\to \cZ$ is self-adjoint since
	\begin{align*}
		(Kz, \delta z)_\cZ
		&= \langle B^{-*} DB^{-1} P z, P \delta z\rangle_{\cZ\times \cZ'}
		= \langle P z,  B^{-*} D B^{-1}  P \delta z\rangle_{\cZ'\times \cZ}
		= (z, K \delta z)_\cZ
	\end{align*}
since $D^*=D$. Hence, the spectrum $\sigma(K)\subset\R_0^+$ is contained in the non-negative reals. This implies $\| (I+\lambda^{-1} K)^{-1}\|_{\cZ\to\cZ} \le 1$ for all $\lambda >0$, so that $\| C^{-1}\|_{\cY'\to\cZ} \le  \| B^{-*}\|_{\cY'\to\cZ} = \frac1\beta$. 
In order to bound $\| L^{-1} \|_{\cW'\to\cW}$, let $(h,g)^T\in\cW$. Then, %since $\| P\|_{\cZ\to\cZ'}\le 1$
\begin{align*}
	\left\| L^{-1} (h,g)^T \right\|_{\cY\times\cZ}
	&\le \textstyle{\frac1\lambda} \| B^{-1} P C^{-1}\|_{\cY'\to\cY} \, \| h- D B^{-1} g\|_{\cY'}
		+ \| B^{-1}\|_{\cZ'\to\cY} \| g\|_{\cZ'}\\
	&\qquad + \| C^{-1}\|_{\cY'\to\cZ}  \| h- D B^{-1} g\|_{\cY'} \\
	&\kern-60pt\le \| C^{-1}\|_{\cY'\to\cZ}  (\textstyle{\frac{\gamma_P}\lambda} 
			\| B^{-1} \|_{\cZ'\to\cY}+1) \| h- D B^{-1} g\|_{\cY'}
		 + \| B^{-1}\|_{\cZ'\to\cY} \| g\|_{\cZ'}\\
	&\kern-60pt\le \textstyle{\frac1\beta} (\textstyle{\frac{\gamma_P}{\lambda\beta}} +1)
		(\| h\|_{\cY'} + \| D\|_{\cY\to\cY'}  \| B^{-1}\|_{\cZ'\to\cY} \| g\|_{\cZ'})
		+  \textstyle{\frac1\beta}  \| g\|_{\cZ'} \\
	&\kern-60pt\le \textstyle{\frac1\beta} (\textstyle{\frac{\gamma_P}{\lambda\beta}} +1) \| h\|_{\cY'}
		+ \big( \textstyle{\frac1{\beta^2}} (\textstyle{\frac{\gamma_P}{\lambda\beta}} +1)
			+ \textstyle{\frac1\beta} \big) \| g\|_{\cZ'} \\
	&\kern-60pt\le \textstyle{\frac1\beta} \max\{ \textstyle{\frac{\gamma_P}{\lambda\beta}} +1,
		\textstyle{\frac1\beta} \big( \textstyle{\frac{\gamma_P}{\lambda\beta}} +1\big) +1\}
		(\| g\|_{\cZ'} + \| h\|_{\cY'})\\
	&\kern-60pt= \Big( \textstyle{\frac1{\beta^2}} \big( \textstyle{\frac{\gamma_P}{\lambda\beta}} +1\big) +1\Big)
			(\| g\|_{\cZ'} + \| h\|_{\cY'})
	= \frac{{\gamma_P}+\lambda\beta + \lambda\beta^2}{\lambda\beta^3} (\| h\|_{\cY'} + \| g\|_{\cZ'}),
\end{align*}
	which proves the claim.
\end{proof}

\begin{corollary}\label{cor:inf_sup_os}
	For the space-time variational formulation \eqref{eq:varformST}, we have
	\begin{align}\label{eq:estinfsup}
		\| L^{-1} \|_{\cW'\to\cW} \le
			2  + \ts{\frac1\lambda},
	\end{align}
	so that the inf-sup-constant of the reduced optimality system \eqref{eq:reducedoptimal} is at least  
	$\frac{\lambda}{1+2\lambda}$.
\begin{proof}
		Since $\beta=1=\| P\|_{\cV\to\cV'}$ the claim follows from \eqref{eq:estLinverse}.
\end{proof}
\end{corollary}

\begin{remark}\label{Rem:optu}
	For the optimal control $\bar{u}$, it holds that
	\begin{align*}
		\| \bar u\|_\cU
		&= \textstyle{\frac1\lambda} \| P\bar{z}\|_\cH
		\le \textstyle{\frac1\lambda} \| \bar{z}\|_\cV 
		\le \textstyle{\frac1\lambda} \| (\bar{y},\bar{z})^T\|_{\cY\times \cZ} 
		\le \textstyle{\frac1\lambda}  \| L^{-1} \|_{\cW'\to\cW} \| (h,g)^T\|_{\cY'\times\cZ'}{.}
	\end{align*}
%	which yields stability and together with well-posedness of \eqref{eq:KKT1}.
\end{remark}

\begin{remark}[Differences to existing space-time methods, continued]\label{Rem:Diff2}
We continue Remark \ref{Rem:Diff1} with highlighting the differences to previous publications.
\begin{compactenum}[(a)]
\item In \cite{MR2407012,MR2874969,MR2861431}, the adjoint problem is derived directly from the variational formulation, which means that the terminal condition is imposed in strong form. This necessarily implies that $\cZ=\cY$ is the space for the adjoint state. Moreover, the same high regularity requirements apply for the solution of the adjoint problem as for the primal one, \cite[Prop.\ 2.3]{MR2407012}. 
%-------------
\item In \cite{MR3343358,MR4076464} the adjoint equation is derived by integration by parts in time, which is possible since $\cZ=\cY$. As in  \cite{MR2407012,MR2874969,MR2861431}, this implies high (and the same) regularity for $y$ and $z$, \cite[La.\ 3.2]{MR3343358}.
%-------------
\item In \cite{MR4223221} the adjoint problem and also the gradient equation is derived in strong form, which is then formulated in space-time variational form. This results in a coupled space-time system where primal and adjoint state have the same regularity. Moreover, since the initial condition is imposed in the primal trial space, the terminal condition is part of the definition of the adjoint trial space. In our case, it holds $\overline{z} \in\cZ$, which allows weaker regularity for the adjoint state. 
\end{compactenum}
\noindent
The differences concerning discretization will be described in the next section.
\end{remark}

%====================================================================
\section{Space-Time Discretization}\label{Sec:4}
%====================================================================
In this section, we are going to describe a conforming discretization of the optimal control problem in space and time. We start by reviewing space-time Petrov--Galerkin methods for parabolic problems from \cite{r.andreev2012,k.urbana.t.patera2012,k.urbana.t.patera2014} and will extend this to a full space-time discretization of the optimal control problem at hand. This leads us to a tensorproduct-type discretization w.r.t.\ time and space variables. Of course, the approach is not restricted to tensorproducts; for example, one could also use unstructured space-time finite elements as in \cite{MR4223221}. However, w.r.t.\ stability and efficient solution of the fully discretized problems, the tensporproduct approach turned out to be very promising, see also \cite{stevenson2021waveletintime,j.henning.etal2019}.

%------------------------------------------------------------------------------------------------------------------------
\subsection{Petrov--Galerkin discretization of the PDE}\label{sec:PetrovGalerkin}
%------------------------------------------------------------------------------------------------------------------------

We consider and construct finite-dimensional spaces $\mathcal{Y}_\delta \subset \mathcal{Y}$ and $\mathcal{Z}_\delta \subset \mathcal{Z}$, where \textendash \ for simplicity \textendash \ we assume that $n_\delta\coloneqq \dim(\cY_\delta)=\dim(\cZ_\delta)$. The Petrov--Galerkin approximation to \eqref{eq:varformST} amounts finding $y_\delta\in\cY_\delta$ such that (for given $u\in\cU$ {to be discretized below})
\begin{align}\label{eq:allgProbPG}
	{b(y_\delta, z_\delta)
	= \langle Ru,z_\delta\rangle_{\cV'\times\cV} 
		+ h(z_\delta) } \qquad \forall z_\delta \in \mathcal{Z}_\delta.
\end{align}
We may think of $\delta=(\Dt,h)$, where $\Dt$ is the temporal and $h$ the spatial mesh width. 
We recall, that there are several ways to select such discrete spaces so that the arising discrete problem is well-posed {and stable in the sense of \eqref{eq:LBB}}. An overview of conditionally and unconditionally stable variants can be found in \cite{r.andreev2012,r.andreev2013A,r.andreev2016A}. In \cite{o.steinbach2015} a finite element approach is described. Moreover, the authors of \cite{k.urbana.t.patera2012,k.urbana.t.patera2014} show that linear ansatz and constant test functions w.r.t.\ time lead to the \emph{Crank--Nicolson} time integration scheme for the special case of homogeneous Dirichlet boundary conditions if the right-hand side is approximated with the trapezoidal rule. A similar approach, but for the case of Robin boundary conditions, is briefly presented in the sequel, where we basically follow \cite{r.andreev2012}. 
It is convenient (and, as we explained above, also efficient from the numerical point of view) to choose the approximation spaces to be of tensorproduct form,   
\begin{align}\label{eq:YdeltaZdelta}
	\mathcal{Y}_\delta = V_{\Dt} \otimes {V}_h, \quad 
	\cZ_\delta = { Q_{\Dt} \otimes {V}_h }
\end{align}
with the temporal subspaces $V_{\Dt} \subset H^1(I)$ und $Q_{\Dt} \subset L_2(I)$ as well as the spatial subspace ${V}_h \subset {V}=H^1(\Omega)$. Our particular choice is as follows: The time interval $I=(0,T)$ is discretized according to
\begin{align*}
	\mathcal{T}_{\Dt} \coloneqq \{0 \eqqcolon t^{(0)} < t^{(1)} < \cdots < t^{(K)} \coloneqq T \} 
	\subset [0,T]\text{,}\quad t^{(k)}=k \cdot \Dt,
\end{align*}
where $K \in \mathbb{N}$ denotes the number of time steps, i.e., $\Dt \coloneqq T/K$ is the time step size. The temporal subspaces $V_{\Dt}$, $Q_{\Dt}$ and the spatial subspace ${V}_h$ read
\begin{align*}
V_{\Dt} \coloneqq  \mathrm{span \ } \Theta_\Dt \subset H^1(I), \,\,\,
Q_{\Dt} \coloneqq  \mathrm{span \ } \Xi_\Dt \subset L_2(I), \,\,\,
{V_h} \coloneqq  \mathrm{span \ } \Phi_h \subset H^1(\Omega)
\end{align*}
with piecewise linear functions $\Theta_\Dt = \{\theta^k \in H^1(I): k= {1},..., K \}$, piecewise constants $\Xi_\Dt = \{\xi^\ell \in L_2(I): \ell=0, ..., K-1 \}$ in time and piecewise linear basis functions in space $\Phi_h = \{\phi_i \in H^1(\Omega): i=1,...,n_h\}$. Doing so, we obtain $\dim(\cY_\delta)=\dim(\cZ_\delta)=n_\delta={K}n_h$.
Such a Petrov--Galerkin discretization for solving \eqref{eq:allgProbPG} amounts determining 
\begin{align}\label{eq:PGapproximation}
	 \mathcal{Y}_\delta \ni y_\delta= \sum\limits_{k={1}}^K \sum\limits_{i=1}^{n_h} y_i^k \theta^k \otimes \phi_i\text{,}
\end{align}
with the coefficient vector ${\by}_\delta \coloneqq [y_1^{{1}},..., y_{n_h}^{{1}},...,y_1^K,..., y_{n_h}^K ]^\top \in \mathbb{R}^{n_\delta}$. 
%The corresponding initial value is given by $y_\delta(0)= \sum_{i=1}^{n_h} y_i^0 \otimes \phi_i = \sum_{i=1}^{n_h} y_i^0 \cdot \phi_i$. 
We are going to derive the arising linear system of equations for \eqref{eq:allgProbPG} 
\begin{align}\label{eq:PrimalProblemMatrix}
	{\bB}_\delta  {\by}_\delta 
	=  {(\bR \bu)_\delta} {+ \bh_\delta}, 
\end{align}
with the stiffness matrix $\bB_\delta\in \mathbb{R}^{n_\delta\times n_\delta}$ and the {vectors $(\bR \bu)_\delta\in \mathbb{R}^{n_\delta}$}, $\bh_\delta \in \mathbb{R}^{n_\delta}$ {to be detailed next}. To this end, we use the basis functions for the test space and obtain for $\ell=0,...,K-1$ and $j=1,...,n_h$ 
\begin{align*}
	b(y_\delta, {\xi^\ell \otimes \phi_j} )
	&\kern-1pt=\kern-3pt \int\limits_I \langle  \dot{y}_\delta(t), \xi^\ell \otimes \phi_j \rangle_{{X}' \times {X}} 
		+ a(y_\delta(t),\xi^l \otimes \phi_j) \diff t 
%		+\left(y_\delta(0), \phi_m \right)_{H} 
		\\
	&\kern-0pt= \sum\limits_{k={1}}^K \sum\limits_{i=1}^{n_h} y_i^k 
		\left[ \langle \dot{\theta}^k \otimes \phi_i, \xi^\ell \otimes \phi_j \rangle_{{X}' \times {X}}
		+ a(\theta^k \otimes \phi_i,\xi^\ell \otimes \phi_j) \diff t \right] 
%		\\ 
%	&\quad
%	+ \sum\limits_{i=1}^{n_h} y_i^0 \left(\phi_i, \phi_m \right)_{H} 
	\\
	&\kern-0pt= \sum\limits_{k={1}}^K \sum\limits_{i=1}^{n_h} y_i^k 
		\left[ \left( \dot{\theta}^k, \xi^\ell \right)_{L_2(I)}\, \left( \phi_i, \phi_j \right)_{H} 
		+\left( \theta^k, \xi^\ell \right)_{L_2(I)}\,  a(\phi_i, \phi_j) \right] 
%		\\ 
%	&\quad
%	+ \sum\limits_{i=1}^{n_h} y_i^0 \left(\phi_i, \phi_m \right)_{H} 
	\text{.}
%	f( (\xi^\ell \otimes \phi_j, \phi_m );u  )
%	&=  f_{\ell,j}^{{1}}(u) + f^{{2}}_{m} \text{,} %\left( y_0, \phi_m \right)_{H} \text{,}
\end{align*}%
{Moreover, it holds $(\bR\bu)_\delta\coloneqq [ r^\ell_j(u)]_{\ell=0,...,K-1; j=1,...,n_h}\in\R^{n_\delta}$ and \linebreak$\bh_\delta \coloneqq [ h^\ell_j]_{\ell=0,...,K-1; j=1,...,n_h}\in\R^{n_\delta}$, where $r^\ell_j(u) \eqqcolon \langle Ru, \xi^\ell \otimes \phi_j\rangle_{\cV'\times\cV}$ and $h^\ell_j \eqqcolon h(\xi^\ell \otimes \phi_j)= (\eta, \xi^\ell \otimes \phi_j )_{\cG}$. The control $u$ will be discretized below}. 
In order to derive a compact form, we introduce a number of matrices
\begin{alignat*}{3} 
	&\underline{C}^{\mathrm{time}}_\Dt \coloneqq \left[ c_{k,\ell} \right]_{k={1}, \ell=0}^{K, K-1} 
	&&\ \in \mathbb{R}^{{K} \times K} 
	&& \qquad \text{with } \qquad c_{k,\ell} \coloneqq (\dot{\theta}^k, \xi^\ell )_{L_2(I)} ,\\[8pt]
	&\underline{N}^{\mathrm{time}}_\Dt \coloneqq \left[ n_{k,\ell} \right]_{k={{1}}, \ell=0}^{K, K-1} 
	&&\ \in \mathbb{R}^{{K} \times K} 
	&& \qquad \text{with } \qquad n_{k,\ell} \coloneqq (\theta^k, \xi^\ell )_{L_2(I)}, \\[8pt]
	&\underline{M}^{\mathrm{time}}_\Dt \coloneqq \left[ m_{k,\ell} \right]_{k=0, \ell=0}^{K-1, K-1} 
	&&\ \in \mathbb{R}^{K \times K} 
	&& \qquad \text{with } \qquad m_{k,\ell} \coloneqq (\xi^k, \xi^\ell )_{L_2(I)}, \\[8pt]
	&\underline{A}^{\mathrm{space}}_h \coloneqq \left[ a_{i,j} \right]_{i,j=1}^{n_h} 
	&&\ \in \mathbb{R}^{n_h \times n_h} 
	&& \qquad \text{with } \qquad a_{i,j} \coloneqq a(\phi_i, \phi_j ) ,\\[8pt]
	&\underline{M}^{\mathrm{space}}_h \coloneqq \left[ m_{i,j} \right]_{i,j=1}^{n_h} 
	&&\ \in \mathbb{R}^{n_h \times n_h} 
	&& \qquad \text{with } \qquad m_{i,j} \coloneqq (\phi_i, \phi_j )_{{H}}. 
\end{alignat*}
Based upon this, we obtain 
{
${\bB}_\delta \coloneqq 
		\underline{C}^{\mathrm{time}}_{\Dt} \otimes \underline{M}^{\mathrm{space}}_h 
			+ \underline{N}^{\mathrm{time}}_\Dt \otimes \underline{A}^{\mathrm{space}}_h
	\in \mathbb{R}^{n_\delta \times n_\delta}$.}

{
\begin{remark}\label{Rem:discInfSup}
	There are several uniformly inf-sup stable discretizations available, see e.g.\ \cite{r.andreev2013A}. For the above case, there is even an optimal discretization, i.e., where the inf-sup constant is unity, \cite{k.urbana.t.patera2012,k.urbana.t.patera2014}. In any case, we have (and shall assume in the sequel) that \eqref{eq:LBB} holds uniformly in $\delta\to 0$, possibly with discrete norms $\|\cdot\|_{\cY_\delta}$, $\|\cdot\|_{\cZ_\delta}$. 
\end{remark}
}

%------------------------------------------------------------------------------------------------------------------------
\subsection{Discretization of the control}
%------------------------------------------------------------------------------------------------------------------------
So far, we did not yet discretize the control $u\in \cU= L_2(I;H)$. {A natural choice seems to be $\cU_\delta \coloneqq  Q_\Dt \otimes V_h=\cZ_\delta$, but other choices are possible as well. Thus,} 
we consider %the discretized control
{
\begin{equation}\label{eq:udelta}
	u_\delta\coloneqq  \sum_{k=0}^{K-1} \sum_{i=1}^{{n}_h} u_i^k\, \xi^k\otimes\phi_i,
	\qquad
	\bu_\delta \coloneqq  [u_i^k]_{k=0,...,K-1;\, i=1,...,n_h}\in\R^{K{n}_h}.
\end{equation}}%
The next step is to detail {$(\bR\bu)_\delta$} based upon this discretization. {We obtain for $\ell = 0,...,K-1$ and $j=1,..., n_h$}
\begin{align}
{r^\ell_j(u_\delta)} &=
	{\langle Ru_\delta, \xi^\ell \otimes \phi_j\rangle_{\cV'\times\cV}} 
	= \sum_{k=0}^{K-1} \sum_{i=1}^{{n}_h} u_i^k\, (\xi^k, \xi^\ell)_{L_2(I)}\, 
	{\langle}{R\phi_i}, \phi_j{\rangle_{V'\times V}} \nonumber\\
	&= [(\underline{M}^{\mathrm{time}}_\Dt\otimes\underline{N}^{\mathrm{space}}_h) \bu_\delta]_{\ell, j}
	{\eqqcolon [ \tilde\bM_\delta \bu_\delta]_{\ell, j}},
		\label{eq:DefbMdelta}
\end{align}
where $\underline{M}^{\mathrm{time}}_\Dt \in \R^{K\times K}$ is the identity (for piecewise constants) as introduced above and 
$\underline{N}^{\mathrm{space}}_h \coloneqq \left[ n_{i,j} \right]_{i=1,j=1}^{n_h,{n}_h}$ with $n_{i,j} \coloneqq {\langle R\phi_i, \phi_j\rangle_{V'\times V}}$\footnote{{For the common case $R=I$, we get $n_{i,j} = (\phi_i, \phi_j)_H$, i.e., $\underline{N}^{\mathrm{space}}_h=\underline{M}^{\mathrm{space}}_h$.}}.   
{Putting everything together, the discretized version of the primal problem \eqref{eq:PrimalProblemMatrix} reads
\begin{equation}\label{eq:PrimalProblemMatrixFull}
	{\bB}_\delta  {\by}_\delta 
	-\tilde{\bM}_\delta  {\bu}_\delta 
	= \bh_\delta.
\end{equation}}%
For later reference, we note that
\begin{align*}
	\| u_\delta\|_{\cH}^2
	= \sum_{k,\ell=0}^{K-1} \sum_{i,j=1}^{{n}_h} u_i^k\, u_j^\ell\, (\xi^k, \xi^\ell)_{L_2(I)}\, (\phi_i, \phi_j)_H 
	&= \bu_\delta^\top\, (\underline{M}^{\mathrm{time}}_\Dt\otimes{\underline{{M}}^{\mathrm{space}}_h})\, \bu_\delta\\
	&\eqqcolon \bu_\delta^\top \,{\bM}_\delta\,  \bu_\delta.
\end{align*}

\begin{remark}
	We stress the fact that we could use any other suitable discretization of the control, both w.r.t.\ time and space, in particular including adaptive techniques {or a discretization arising from implicitly utilizing the optimality conditions and the discretization of the state and adjoint equation, see e.g. \cite{Hinze2005}}. 
\end{remark}

%------------------------------------------------------------------------------------------------------------------------
\subsection{Petrov--Galerkin discretization of the adjoint problem}\label{sec:adjointSystemSpaceTime}
%------------------------------------------------------------------------------------------------------------------------
We are now going to derive the discrete form of the adjoint problem \eqref{eq:KKT1:b} or \eqref{eq:KKT2:b}. Since this problem involves the adjoint operator, it seems reasonable to use the same discretization, so that the {(matrix-vector form of the)} discrete problem amounts finding ${\bz_\delta}\in\R^{n_\delta}$ such that {(for given $y_\delta\in\cY_\delta$)}
\begin{align}\label{eq:AdjointProblemMatrix}
	\bB_\delta^\top {\bz_\delta} + {\bd_\delta(y_\delta) = \bg_\delta} ,
\end{align}	
{with $\bd_\delta(y_\delta)\in\R^{n_\delta}$, $\bg_\delta\in\R^{n_\delta}$, i.e., $d(y_\delta, \delta y_\delta)
					+b(\delta y_\delta ,z_\delta) 
					= g(\delta y_\delta)$ 
	for all $\delta y_\delta\in\cY_\delta$. Note, that} the stiffness matrix is the transposed of the stiffness matrix of the primal problem. The unknown coefficient vector {$\bz_\delta\in\R^{Kn_h}$ reads}
\begin{align}\label{eq:zdelta}
	\cZ_\delta \ni {z_\delta} 
	&
	{= \sum_{k=0}^{K-1}\sum_{i=1}^{n_h} z_i^k\, \xi^k\otimes\phi_i}, \qquad
	\bz_\delta \coloneqq  [z_i^k]_{k=0,...,K-1;\, i=1,...,n_h}\in\R^{K n_h}.
\end{align}
Let us now detail the {remaining terms $\bd_\delta(y_\delta)=[d_j^\ell(y_\delta)]_{\ell=1,...,K;\, j=1,...,n_h}\in\R^{n_\delta}$, $\bg_\delta=[g_j^\ell]_{\ell=1,...,K;\, j=1,...,n_h}\in\R^{n_\delta}$}. 
{For $\ell=1,...,K$ and $j=1,...,n_h$, we get
\begin{align*}
g^\ell_j &= g(\theta^\ell\otimes \phi_j) = (y_d, \theta^K(T)\otimes \phi_j)_H = \theta^K(T) \cdot (y_d, \phi_j)_H.
\end{align*}
Further, we abbreviate the coefficient vector of $y_\delta(T)$ in terms of the basis $\Phi_h$ as ${{\by}_{\delta}^K} \coloneqq [y_1^K,..., y_{n_h}^K ]^\top$ so that by \eqref{eq:PGapproximation} we obtain for $\ell=1,...,K$ and $j=1,...,n_h$ 
\begin{align*}
d_j^\ell(y_\delta) 
&= d(y_\delta, \theta^\ell\otimes \phi_j) 
= (y_\delta(T), \theta^K(T)\otimes \phi_j)_H \\
&= \theta^\ell(T)\cdot \theta^K(T) \cdot \sum\limits_{i=1}^{n_h} y_i^K  (\phi_i, \phi_j)_H 
= \delta_{\ell,K}\cdot [\underline{M}^{\mathrm{space}}_h \by_\delta^K ]_{j} \eqqcolon [\bD_\delta \by_\delta]_{\ell,j},
\end{align*}
where $\delta_{\ell,K}$ denotes the discrete Kronecker delta and we introduce
\begin{align*}
\bD_\delta\coloneqq \begin{pmatrix} 
	\boldsymbol{0} & \boldsymbol{0} \\ 
	\boldsymbol{0} & \underline{M}^{\mathrm{space}}_h 
	\end{pmatrix}\in\R^{n_\delta\times n_\delta}.
\end{align*}
We note that it holds $\theta^\ell(T)=\delta_{\ell,K}$, i.e., $\theta^\ell(T)=0$ for $\ell=1,...,K-1$ and $\theta^K(T)=1$. With this notation at hand, the fully discretized version of the adjoint problem \eqref{eq:AdjointProblemMatrix} reads
\begin{align}\label{eq:DualProblemMatrixFull}
	\bB_\delta^\top \bz_\delta + \bD_\delta\by_\delta = \bg_\delta.
\end{align}
}

%------------------------------------------------------------------------------------------------------------------------
\subsection{Petrov--Galerkin discretization of the gradient equation}
%------------------------------------------------------------------------------------------------------------------------
In order to obtain a discrete version of the gradient equation we test \eqref{eq:KKT2:c} with the basis functions of $\cU_\delta$, namely $\lambda \left( u_\delta, \delta u_\delta \right)_{\cH}- \langle R \delta u_\delta, z_\delta\rangle_{\cV'\times\cV}= 0$ for all $\delta u_\delta\in\cU_\delta$. Recalling the discretizations \eqref{eq:udelta} and \eqref{eq:zdelta} of $u_\delta$ and $z_\delta$, respectively, we obtain for $\ell=0,...,K-1$ and $j=1,...,{n}_h$
\allowdisplaybreaks
\begin{align*}
	0
	&= \lambda \left( u_\delta, \xi^\ell\otimes\phi_j \right)_{\cH}
				- \langle z_\delta , \xi^\ell\otimes R\phi_j\rangle_{\cV\times\cV'} \\
	&=\lambda  \sum_{k=0}^{K-1} \sum_{i=1}^{{n}_h} 
		u_i^k\, (\xi^k\otimes\phi_i, \xi^\ell\otimes\phi_j )_{\cH}
		- \sum_{k=0}^{K-1}\sum_{i=1}^{n_h} z_i^k\, 
			\langle \xi^k\otimes\phi_i, \xi^\ell\otimes R\phi_j \rangle_{\cV\times\cV'} \\
	&= \lambda  \sum_{k=0}^{K-1} \sum_{i=1}^{{n}_h}  u_i^k \,
			(\xi^k,\xi^\ell)_{L_2(I)}\, (\phi_i, \phi_j)_H
		- \sum_{k=0}^{K-1}\sum_{i=1}^{n_h} z_i^k\,(\xi^k,\xi^\ell)_{L_2(I)}\, 
			\langle R\phi_j, \phi_i\rangle_{V'\times V} \\
	&= \lambda [(\underline{M}^{\mathrm{time}}_\Dt\otimes \underline{{M}}^{\mathrm{space}}_h) \bu_\delta]_{\ell, j} - [(\underline{M}^{\mathrm{time}}_\Dt\otimes(\underline{N}^{\mathrm{space}}_h)^\top) \bz_\delta]_{\ell, j} \\
	&=\lambda [{\bM}_\delta \bu_\delta]_{\ell, j} - [ \tilde\bM_\delta^\top \bz_\delta]_{\ell, j}	.	
\end{align*}
Then, the discrete version of the gradient equation \eqref{eq:KKT2:c} reads
 \begin{align}\label{eq:OptimalityMatrixFull}
	\lambda{\bM}_{\delta} \bu_\delta - \tilde\bM_\delta^\top\bz_\delta = \boldsymbol{0}.
\end{align}
We note, that ${\bM}_{\delta}$ is a square mass matrix, i.e., invertible. For $R=I$, we have $\tilde\bM_\delta^\top={\bM}_{\delta}$, so that $\bu_\delta=\lambda^{-1} \bz_\delta$.

%------------------------------------------------------------------------------------------------------------------------
\subsection{The discrete optimality system}\label{Sec:discoptsys}
%------------------------------------------------------------------------------------------------------------------------
We can now put all pieces together and detail the discrete version of the first order optimality system \eqref{eq:KKT2}, namely
\begin{subequations}\label{eq:discKKT}
\allowdisplaybreaks
		\begin{align}
\label{eq:discKKT:a}
			b(\overline{y}_\delta,\delta z_\delta) 
				- \langle R\overline{u}_\delta,\delta z_\delta\rangle_{\cV'\times\cV}	&= h(\delta z_\delta) 
					&& \forall \delta z_\delta\in\cZ_\delta, \\\label{eq:discKKT:b}
			d( \overline{y}_\delta, \delta y_\delta)
					+b(\delta y_\delta,\overline{z}_\delta) 
					&= g(\delta y_\delta) 
					&& \forall \delta y_\delta\in\cY_\delta, \\\label{eq:discKKT:c}
			\lambda ( \overline{u}_\delta, \delta u_\delta)_\cH 
					- \langle R\delta u_\delta, \overline{z}_\delta\rangle_{\cV'\times\cV}
					&= 0
					&& \forall \delta u_\delta \in \cU_\delta.
		\end{align}
\end{subequations}

Recalling \eqref{eq:PrimalProblemMatrixFull}, \eqref{eq:DualProblemMatrixFull} and \eqref{eq:OptimalityMatrixFull}, the discrete first order optimality system \eqref{eq:discKKT} can be written in matrix form as
	\begin{align*}
		\begin{pmatrix} 
			\bD_\delta & \bB_\delta^T & \boldsymbol{0}  \\ 
			\bB_\delta & \boldsymbol{0} & -\tilde\bM_\delta  \\
			\boldsymbol{0} & -\tilde\bM_{\delta}^\top & \lambda\,{\bM}_{\delta}   
			\end{pmatrix}
		\begin{pmatrix} \by_\delta \\ \bz_\delta \\ \bu_\delta \end{pmatrix}
		= \begin{pmatrix} \bg_\delta \\ \bbf_\delta \\ \boldsymbol{0}  \end{pmatrix},
	\end{align*}
where all involved matrices have tensorproduct structure. In view of \eqref{eq:OptimalityMatrixFull}, i.e, $\lambda{\bM}_{\delta} \bu_\delta = \tilde\bM_{\delta}^\top \bz_\delta$, we can easily eliminate the variable $\bu_\delta$ and obtain the reduced system
	\begin{align}\label{eq:reddisoptsys}
		\bL_\delta
		\begin{pmatrix} 
			\by_\delta \\
			\bz_\delta   
		\end{pmatrix}
		= \begin{pmatrix} 
		\bg_\delta \\
		\bbf_\delta  
		\end{pmatrix},
		\qquad
		\bL_\delta :=
		\begin{pmatrix} 
			 \bD_\delta & \bB_\delta^T  \\
			 \bB_\delta 
			 & -\textstyle{\frac1\lambda} \tilde\bM_\delta \bM_{\delta}^{-1} \tilde\bM_\delta 
		\end{pmatrix}
	\end{align}
which is a discretized version of \eqref{eq:reducedoptimal}. All involved matrices are tensorproducts and for our choice, we have $\bM_{\delta}=\tilde\bM_\delta$. Set $\gamma:= \| \tilde\bM_\delta\bM_{\delta}^{-1} \tilde\bM_\delta\|= \|\bM_\delta\|$, which is bounded uniformly in $\delta\to 0$.

\begin{theorem}[Well-posedness of the discrete optimality system]\label{Thm:DiscreteOptSystem}
	Assume that the discrete inf-sup condition \eqref{eq:LBB} holds. Then, the discrete first order optimality system \eqref{eq:discKKT} is well-posed for all $\lambda>0$ and
	\begin{align*}
		\|\by_\delta\| + \| \bz_\delta\| &\le 
			\ts{\frac{\gamma+\lambda\beta + \lambda\beta^2}{\lambda\beta^3}}
			(\| \bh_\delta\| + \| \bg_\delta\|), 
			&
		\| \bu_\delta\| \le 
			\ts{\frac{\gamma+\lambda\beta + \lambda\beta^2}{\lambda^2\beta^3}}
			(\| \bh_\delta\| + \| \bg_\delta\|). 
	\end{align*}
\end{theorem}
\begin{proof}
	We can apply Theorem \ref{Thm:OptSystem} for the reduced discrete optimality system \eqref{eq:reddisoptsys}. Following the lines of its proof, \eqref{eq:estLinverse} ensures that 
		\begin{align}\label{eq:estLdeltainverse}
			\| \bL_\delta^{-1} \| 
			\le
			\ts{\frac{\gamma+\lambda\beta + \lambda\beta^2}{\lambda\beta^3}}, 
	\end{align}
	so that the reduced system is uniformly invertible. Adapting Remark \ref{Rem:optu} for the discrete case yields $\|\bu_\delta\| \le \textstyle{\frac1\lambda} \| \bL_\delta^{-1} \| (\| \bh_\delta\| + \| \bg_\delta\|)$.
\end{proof}

\begin{remark}[Differences to existing space-time methods, continued]\label{Rem:Diff3}
We continue Remarks \ref{Rem:Diff1} and \ref{Rem:Diff2} with highlighting the differences to previous publications, now concerning the discretization.\\
To summarize our approach, we start by an optimally stable Petrov--Galerkin discretization of the state equation based upon tensorproducts, which can be chosen to be equivalent to a Crank--Nicolson time stepping method. In a second step, we chose an appropriate discretization of the control. This automatically yields a Petrov--Galerkin discretization  of the adjoint equation and the gradient equation. Putting everything together results in a stable discretization of the optimality system along with a priori and a posteriori error estimates.
\begin{compactenum}[(a)]
\item \cite{MR2407012,MR2874969,MR2861431} suggest semi-discretizations for the state by discontinuous Galerkin methods. Since there $\cZ=\cY$, this can also be used for the adjoint state. Stability and approximation results are then given.
%-------------
\item \cite{MR3343358,MR4076464} uses a Petrov--Galerkin method with temporal discontinuous trial functions and continuous test functions for primal and adjoint problems. The control is not discretized, but treated in a variational manner.
%-------------
\item In \cite{MR4223221}, the derivation yields a $2\times 2$ saddle point problem for primal and dual state similar to \eqref{eq:KKT3}, which is discretized in a similar fashion as in our approach for the primal state equation. This means that also here primal and dual states use discretizations of the same order, which is different from our approach.\\
Moreover, \cite{MR4223221} uses an unstructured space-time discretization, whereas we suggest a tensorproduct approach. However, the tensorproduct discretization was here mainly chosen to allow the use of efficient solvers and can easily be replaced by other discretizations as well.
\end{compactenum}
\end{remark}

%------------------------------------------------------------------------------------------------------------------------
\subsection{Error analysis}%
%------------------------------------------------------------------------------------------------------------------------

\begin{corollary}[A priori estimate]\label{Cor:Apriori}
	Applying the Xu--Zikatanov Lemma \cite{MR1971217} yields a quasi-best approximation statement, i.e., 
		\begin{align}\label{eq:errorbounds}
		&\| y-y_\delta\|_\cY + 
		\|z-z_\delta\|_\cZ +
		\|u-u_\delta\|_\cU
			\le \nonumber\\
		&\le\max\{1,{\textstyle{\frac1\lambda}}\} 
				\ts{\frac{1+\lambda\beta + \lambda\beta^2}{\lambda\beta^3}}
			\left(
			\inf_{\tilde y_\delta\in\cY_\delta} \| y-\tilde y_\delta\|_\cY
			+ \inf_{\tilde z_\delta\in\cZ_\delta} \| z-\tilde z_\delta\|_\cZ
			+ \inf_{\tilde u_\delta\in\cU_\delta} \| u-\tilde u_\delta\|_\cU
			\right).
			\nonumber
		\end{align}
\end{corollary}

Using the above described discretization for $\cY_\delta$, $\cZ_\delta$ and $\cU_\delta$, we get an error of order $\cO(\max\{ h,\Delta t\})$ in the prescribed norms, which can easily be improved by using higher order discretizations (if the solution is sufficiently regular).

\begin{corollary}[A posteriori estimate]
	It holds that
		\begin{align*}%\label{eq:posteriori}
		\| y-y_\delta\|_\cY + 
		\|z-z_\delta\|_\cZ +
		\|u-u_\delta\|_\cU
		\le 
		\left( 2+{\textstyle{\frac1\lambda}}\right) \| r_\delta\|_{\cY'\times\cZ'\times\cU'},
		\end{align*}
		where $r_\delta := (g - D y_\delta, h - B y_\delta + R u_\delta, \lambda u_\delta - R^* z_\delta)^\top$ is the \emph{residual} of the optimality system.
\end{corollary}

The latter estimate allows us to use residual-based error estimates, which are e.g.\ particularly relevant for the reduced basis method in the case of parameter-dependent problems, see e.g.\ \cite{HeRoSt16}.

%------------------------------------------------------------------------------------------------------------------------
\subsection{Discretization of the cost function}%
%------------------------------------------------------------------------------------------------------------------------
Finally, we detail the space-time discretization of the cost function, i.e., 
\begin{align*}
	\hatbJ_\delta(\bu_\delta) 
	&%\coloneqq  J_\delta(y_\delta,u_\delta) 
	\coloneqq  
	\ts{\frac{1}{2}}
		\|  y_\delta(T) - y_{d,h} \|_H^2   + 
		\ts{\frac{\lambda}{2}}  \|  u_\delta \|_\cH^2 \\
	&\kern-0pt
	= \ts{\frac12} (\theta^K(T)\, \by_{\delta}^K  - \by_{d,h})^\top \underline{M}^{\mathrm{space}}_h (\theta^K(T)\, \by_{\delta}^K  - \by_{d,h})
		+ \ts{\frac\lambda2} \bu_\delta^\top\,{\bM}_\delta \, \bu_\delta,
\end{align*}
where $y_{d,h} = \sum_{i=1}^{n_h} y_{d,i}\, \phi_i$ with the coefficient vector $\by_{d,h} \coloneqq  (y_{d,i})_{i=1,...,n_h}\in\R^{n_h}$ is a discretization of $y_d$.\footnote{{Doing so, we have that $y_{d,h} \in V_h$ (i.e., a piecewise linear approximation), which is useful for our experiments. We could of course have also used a piecewise constant approximation.}} 

We solve the optimal control problem by numerically solving the reduced $2 \times 2$ block linear system arising from the optimality system \eqref{eq:reddisoptsys}.

%====================================================================
\section{Numerical Results}\label{Sec:5}
%====================================================================
In this section, we present some results of our numerical experiments. {We follow two main goals: (1) We make quantitative comparisons concerning the inf-sup-stability of the optimality system and (2) we} 
compare the above presented space-time variational approach with the standard semi-discretization (see also \cite{s.glasa.mayerhoferk.urban2017} for such comparisons for parabolic problems).  We do not compare with other state-of-the-art methods as we are mainly interested in investigating the effect of simultaneous space-time discretization.  
In order to make the comparison fair, we chose an all-at-once method for the semi-discrete framework so that the reduced discrete optimality system is built in a similar manner in both approaches. Moreover, we used the Crank--Nicolson scheme for the semi-discrete problem since our choice for trial and test spaces for the primal problem is equivalent to this time-stepping scheme, \cite{k.urbana.t.patera2012,k.urbana.t.patera2014}. Thus, in the semi-discrete setting, primal and dual problems amount for a comparable number of operations, with a stability issue for the dual problem, of course. 
{Note that, in the semi-discrete case, the Crank--Nicolson scheme for the adjoint problem (i.e., $\partial_t z + \Laplace z = 0;\ z(T) = y(T)- y_d$) is backward in time.}

All results were obtained with \textsc{Matlab} R2020b on a machine with a quad core with $2.7$ GHz and $16$ GB of RAM.

{
\subsection{Discrete inf-sup constant}
We start by computing the discrete inf-sup constant of the optimality system and compare that with the bound \eqref{eq:estinfsup} in Corollary \ref{cor:inf_sup_os}. We report the data for a 1d example on $I \times \Omega =(0,1) \times (-1,1)$ with $\mu(x) = x^2 + 0.1$ for $n_h = 40$, $K= 80$, but stress that the results are representative also for other examples.

First, we investigate the dependence of the inf-sup constant w.r.t.\ the regularization parameter $\lambda$. In Figure \ref{plot_inf_sup_lambda}, we show the computed discrete inf-sup constant in comparison with the lower bound $(2+{\textstyle{\frac1\lambda}})^{-1}$. We observe the same quantitative behaviors of both curves and see that our bound seems to be almost sharp for increasing values of $\lambda$. In particular, we see the optimality (inf-sup is unity) already for $\lambda=10^{-2}$ and larger. For small values of $\lambda$, the bound is too pessimistic by almost two orders of magnitude.
\begin{figure}[ht]
	% !TEX root = ./BRU.tex
\begin{center}
\begin{tikzpicture}

    \begin{loglogaxis}
    	[
	name = plot1,
	width=0.6\textwidth,
	height=0.4\textwidth,
	axis lines=left,
	xlabel={$\lambda$},
	every axis x label/.style={
    		at={(ticklabel* cs:1.05)},
    		anchor=west,
		},
	legend pos = south east,
	ymin= 8e-5, ymax=2,
	]
    	\addplot[style=solid, color=blue, line width=1pt] 
    		table [x=lambda,y=beta_OS]{./inf_sup_os_lambda.dat}; 
    	\addlegendentry{discrete inf-sup}
    	\addplot[style=densely dashdotdotted, color=black, line width=1pt]table [x=lambda,y=lower_bound]{./inf_sup_os_lambda.dat}; 
    	\addlegendentry{lower bound} 
    \end{loglogaxis}		    
	    	    						
\end{tikzpicture}
\end{center}
	\caption{Discrete inf-sup constant of the optimality system for different values of $\lambda$.} \label{plot_inf_sup_lambda}
\end{figure}

Next, we fix $\lambda=10^{-2}$ and investigate the dependence of the discretization. The results are presented in Figure \ref{plot_inf_up_n_K}. On the left, we fix $K=60$ and vary $n_h$, whereas on the right, we choose $n_h=60$ and modify $K$. We see that the lower bound is in fact pessimistic, stability improves as $K$ increases and worsens for $n_h$ -- as to be expected. However, we can confirm uniform stability (i.e., for all choices of $n_h$ and $K$) in all cases.
\begin{figure}[ht]
	% !TEX root = ./BRU.tex
\begin{minipage}{0.49\textwidth}
\begin{tikzpicture}
    \begin{semilogyaxis}
    	[
	width=\textwidth,
	height=0.8\textwidth,
	axis lines=left,
	xlabel={$n_h$},
	legend pos = north east,
	ymin= .5e-2,
	ymax=6e0,
	title ={$K = 60$}, 
	]
    	\addplot[style=solid, color=blue, line width=1pt] 
    		table [x=n_h,y=beta_OS]{./inf_sup_os_n.dat}; 
    	\addlegendentry{discrete inf-sup}
    	\addplot[style=densely dashdotdotted, color=black, line width=1pt]table [x=n_h,y=lower_bound]{./inf_sup_os_n.dat}; 
    	\addlegendentry{lower bound} 
    \end{semilogyaxis}					
\end{tikzpicture}
\end{minipage}
\begin{minipage}{0.49\textwidth}
\begin{tikzpicture}
    \begin{semilogyaxis}
    	[
		title ={$n_h = 60$},	
	width=\textwidth,
	height=0.8\textwidth,
	axis lines=left,
	xlabel={$K$},
	legend pos = north west,
	ymin= .5e-2,
	ymax=6e0,
	]
    	\addplot[style=solid, color=red, line width=1pt] 
    		table [x=K,y=beta_OS]{./inf_sup_os_K.dat}; 
    	\addlegendentry{discrete inf-sup}
    	\addplot[style=densely dashdotdotted, color=black, line width=1pt]table [x=K,y=lower_bound]{./inf_sup_os_K.dat}; 
    	\addlegendentry{lower bound} 
    \end{semilogyaxis}	
    				
\end{tikzpicture}
\end{minipage}
	\caption{Discrete inf-sup constant of the optimality system for $\lambda=10^{-2}$ and different values of $n_h$ and $K$.} 	\label{plot_inf_up_n_K}
\end{figure}

\subsection{Comparison of space-time and semi-discrete methods}
Our next aim is to compare our space-time method  with the classical time-stepping. As already pointed out earlier, we choose the data in such a way that the results are in fact comparable.

\subsubsection{One-dimensional example}
We start by Problem \ref{prob:classicalFormulationOptContr} on $I \times \Omega =(0,1) \times (-1,1)$ for $\mu(x) \coloneqq x + 1.3$, boundary data $\eta(t,x) \coloneqq  \frac{\mathrm{sign}(x) 50 t^2}{\cosh(50x)} + (1.3 + x) \tanh(50 x) t^2$ and desired state $y_d(x) = \tanh(50 x)$. Again, we note that we got comparable results also for other data. 
We compare the value of the objective function that we reach by solving the optimality system with the two approaches. The results are shown in Figure \ref{plot_1D_obj_func} for two values of the regularization parameter $\lambda$. We show the value for increasing number $K$ of time steps and two different spatial discretizations, namely $n_h=101$ and $n_h=1001$.  First, we observe that the overall performance is independent of the choice of $\lambda$. Next, we see that both methods converge to the same value of the objective function as $K$ increases. However, the huge benefit of the space-time setting shows off, namely that we reach an almost optimal value also for very coarse temporal discretizations, which offers significant computational savings. 
\begin{figure}[ht]
	% !TEX root = ./BRU.tex

% filename(path) of datfile

\begin{minipage}{0.49\textwidth}
\renewcommand{\fnameone}{./1D_tanh_n_101_lam_1em1.dat}
\renewcommand{\fnametwo}{./1D_tanh_n_1001_lam_1em1.dat}

\begin{tikzpicture}    
    \begin{semilogxaxis}
    [
      width =\textwidth, 
      height =\textwidth,
      xlabel ={$K$},
      ylabel ={Objective function $J(y,u)$},
      xmin=0,xmax=1010,
      title = {$\lambda = 10^{-1}$},
      ]
    	\addplot[style=densely dashdotdotted,
    	         color=blue,
    	         mark=diamond,
    	         mark options={solid},
    	         line width=1pt] table [x=K,y=J_SD]{\fnameone};  
	\addlegendentry{sd, $n_h=101$}
    	\addplot[style=densely dashdotdotted,
    	         color=red,
    	         mark=o,
    	         mark options={solid},
    	         line width=1pt] table [x=K,y=J_ST]{\fnameone};  	
	\addlegendentry{st, $n_h=101$}

    	\addplot[style=solid,
    	         color=blue,
    	         mark=diamond,
    	         mark options={solid},
    	         line width=1pt] table [x=K,y=J_SD]{\fnametwo};  	
	\addlegendentry{sd, $n_h=1001$}
    	\addplot[style=solid,
    	         color=red,
    	         mark=o,
    	         mark options={solid},
    	         line width=1pt] table [x=K,y=J_ST]{\fnametwo};  	
 	\addlegendentry{st, $n_h=1001$}
  \end{semilogxaxis}
  
\end{tikzpicture}
\end{minipage}
\begin{minipage}{0.49\textwidth}

\renewcommand{\fnameone}{./1D_tanh_n_101_lam_1em3.dat}
\renewcommand{\fnametwo}{./1D_tanh_n_1001_lam_1em3.dat}

\begin{tikzpicture}
    
    \begin{semilogxaxis}
    [
      width =\textwidth, 
      height =\textwidth,
      xlabel ={$K$},
      ylabel ={Objective function $J(y,u)$},
      xmin=0,xmax=1010, 
      title = {$\lambda = 10^{-3}$},
      ]
    	\addplot[style=densely dashdotdotted,
    	         color=blue,
    	         mark=diamond,
    	         mark options={solid},
    	         line width=1pt] table [x=K,y=J_SD]{\fnameone};  
	\addlegendentry{sd, $n_h=101$}
    	\addplot[style=densely dashdotdotted,
    	         color=red,
    	         mark=o,
    	         mark options={solid},
    	         line width=1pt] table [x=K,y=J_ST]{\fnameone};  	
	\addlegendentry{st, $n_h=101$}

    	\addplot[style=solid,
    	         color=blue,
    	         mark=diamond,
    	         mark options={solid},
    	         line width=1pt] table [x=K,y=J_SD]{\fnametwo};  	
	\addlegendentry{sd, $n_h=1001$}
    	\addplot[style=solid,
    	         color=red,
    	         mark=o,
    	         mark options={solid},
    	         line width=1pt] table [x=K,y=J_ST]{\fnametwo};  	
 	\addlegendentry{st, $n_h=1001$}
  \end{semilogxaxis}
  
\end{tikzpicture}
\end{minipage}
	\caption{1d example, values of the objective function for different discretizations (left: $\lambda= 10^{-1}$, right: $\lambda=10^{-3}$, abbreviations: semi-discrete (sd), space-time (st)).}
	\label{plot_1D_obj_func}
\end{figure}
It is not surprising that this effect is due to the improved stability of the space-time method as we can also see in Figure \ref{fig_stability_st_sd}, where we depict the control for different values of $K$ for $\lambda=10^{-3}$ and $n_h=101$. We can clearly observe the stability issues for the semi-discrete approach in the left column, which do not appear in the space-time context.  
\begin{figure}[ht]
\begin{minipage}{0.49\textwidth}
\includegraphics[width=\textwidth]{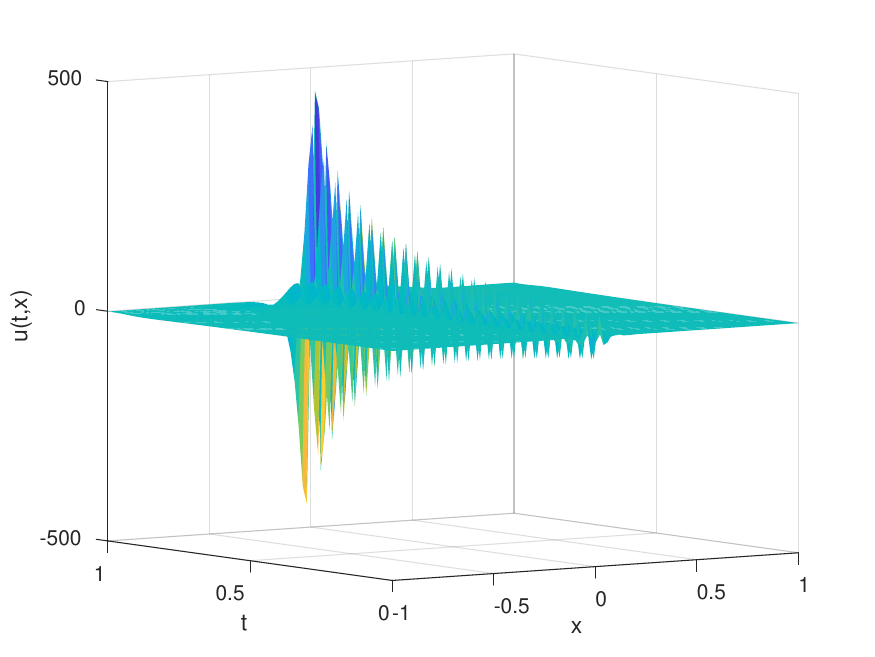}
\end{minipage}
\begin{minipage}{0.49\textwidth}
\includegraphics[width=\textwidth]{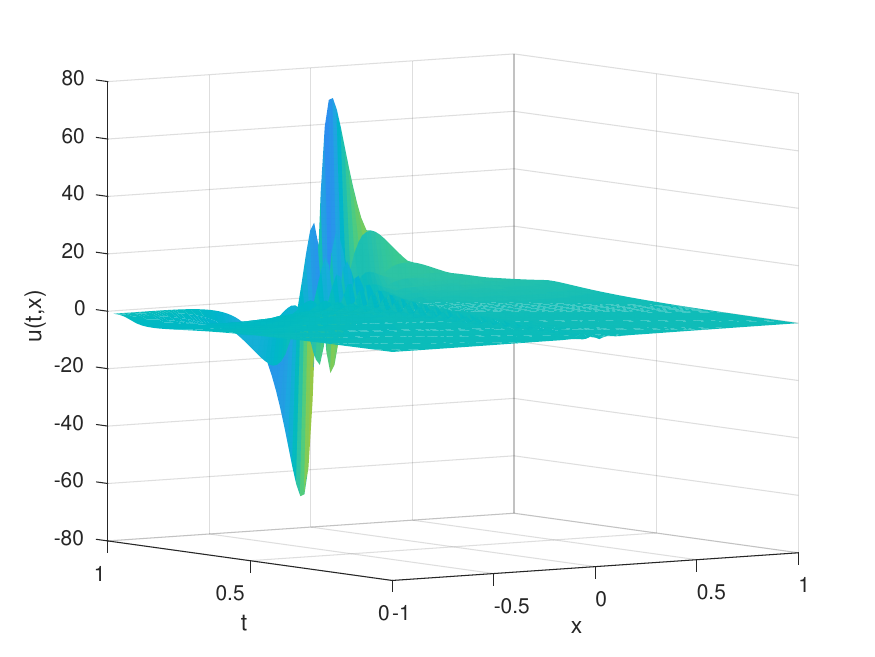}
\end{minipage}

\begin{minipage}{0.49\textwidth}
\includegraphics[width=\textwidth]{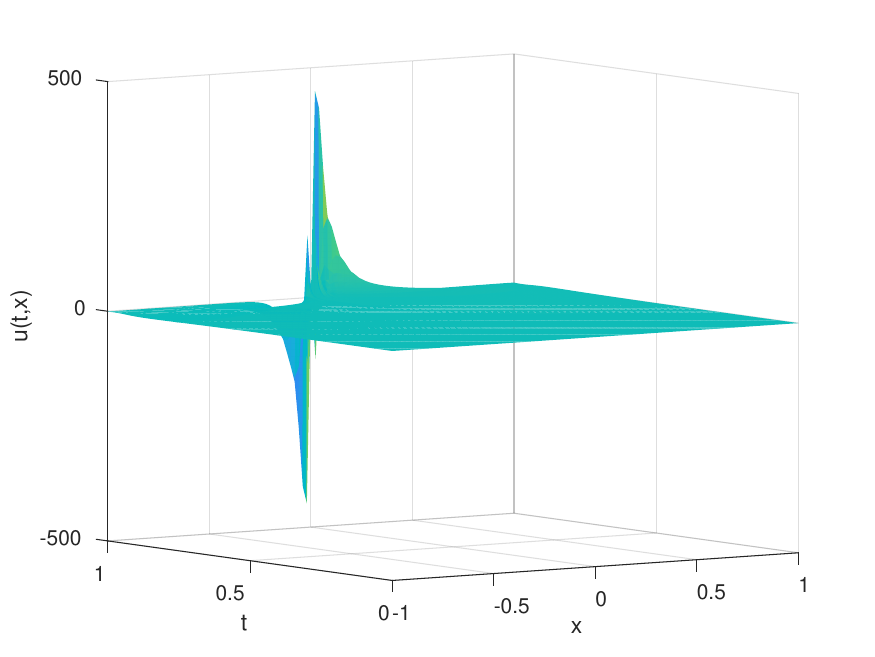}
\end{minipage}
\begin{minipage}{0.49\textwidth}
\includegraphics[width=\textwidth]{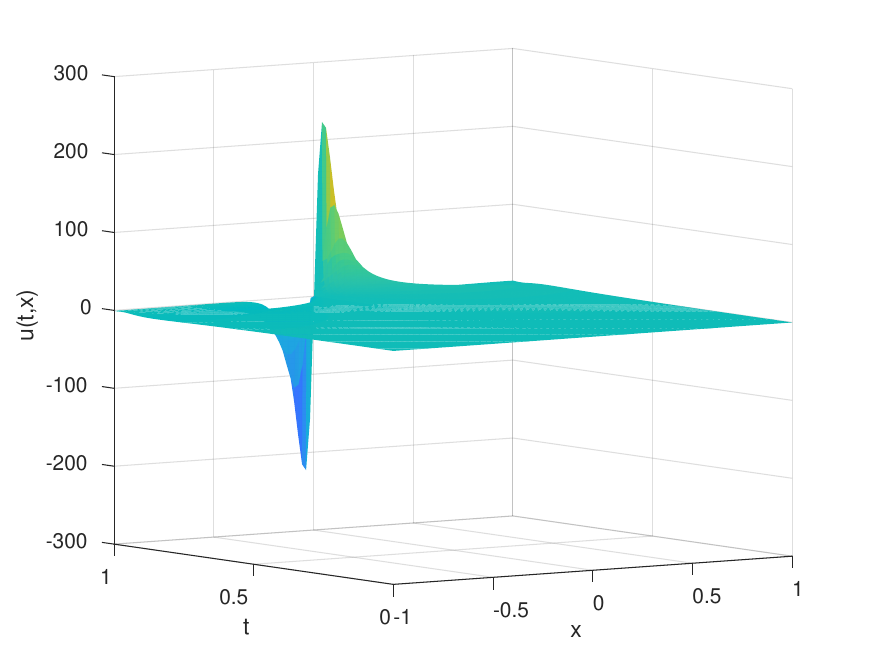}
\end{minipage}
\caption{Optimal control for $\lambda = 10^{-3}$ and $n_h=101$ ({left column:} semi-discrete, {right column:} space-time, {top row:} $K = 50$, {bottom row:} $K=500$).}
\label{fig_stability_st_sd}
\end{figure}

\subsubsection{Higher dimensional examples}
A possible criticism of the space-approach is the fact that the size of the optimality system might significantly grow with increasing space dimension. Hence, we realized both approaches also in 2d and 3d and report the results in the 2d case here.  We do not monitor CPU-time comparisons, but refer e.g.\ to \cite{j.henning.etal2019,henning2021weak}, where such comparisons have been done for space-time variational formulations of the heat and wave equation, respectively. It was shown there, that appropriate tensorproduct solvers in fact yield competitive CPU times for the arising space-time systems. The adaptation of those approaches to the optimality system \eqref{eq:reddisoptsys} is subject to ongoing work, see also Remark \ref{Rem:CPU} below.

Hence, we are going to report results for $I \times \Omega \coloneqq (0,1) \times (0,1)^2$ and boundary data $\mu(x) \coloneqq 0.25 \cosh(x y) +0.25$ along with a compatible function $\eta$. As desired state, we choose $y_d(x) = \tanh(10 (x-0.5)(y-0.5))$. As in the 1d case, we compare the values of the objective function, see Figure \ref{plot_2D_obj_func}. The overall behavior is very similar to the 1d case, namely we get a significant improvement of the space-approach over the semi-discrete one for small number of time steps $K$. Note, that here we use a linear scale for the horizontal axis as opposed to Figure \ref{plot_1D_obj_func}, where the results are shown in logarithmic scale. Moreover, for large values of $\lambda$, we observe the necessity of a sufficiently fine spatial discretization for both methods.
\begin{figure}[ht]
	% !TEX root = ./BRU.tex

% filename(path) of datfile

\begin{minipage}{0.49\textwidth}
\renewcommand{\fnameone}{./2D_tanh_alpha_10_n_54_lam_1em1.dat}
\renewcommand{\fnametwo}{./2D_tanh_alpha_10_n_213_lam_1em1.dat}

\begin{tikzpicture}    
    \begin{axis}
    [
      width =\textwidth, 
      height =\textwidth,
      xlabel ={$K$},
      ylabel ={Objective function $J(y,u)$},
      xmin=0,xmax=405, 
      title = {$\lambda = 10^{-1}$},
      ]
    	\addplot[style=densely dashdotdotted,
    	         color=blue,
    	         mark=diamond,
    	         mark options={solid},
    	         line width=1pt] table [x=K,y=J_SD]{\fnameone};  
	\addlegendentry{sd, $n_h=54$}
    	\addplot[style=densely dashdotdotted,
    	         color=red,
    	         mark=o,
    	         mark options={solid},
    	         line width=1pt] table [x=K,y=J_ST]{\fnameone};  	
	\addlegendentry{st, $n_h=54$}

    	\addplot[style=solid,
    	         color=blue,
    	         mark=diamond,
    	         mark options={solid},
    	         line width=1pt] table [x=K,y=J_SD]{\fnametwo};  	
	\addlegendentry{sd, $n_h=213$}
    	\addplot[style=solid,
    	         color=red,
    	         mark=o,
    	         mark options={solid},
    	         line width=1pt] table [x=K,y=J_ST]{\fnametwo};  	
 	\addlegendentry{st, $n_h=213$}
  \end{axis}
  
\end{tikzpicture}
\end{minipage}
\begin{minipage}{0.49\textwidth}

\renewcommand{\fnameone}{./2D_tanh_alpha_10_n_54_lam_1em3.dat}
\renewcommand{\fnametwo}{./2D_tanh_alpha_10_n_213_lam_1em3.dat}

\begin{tikzpicture}
    
    \begin{axis}
    [
      width =\textwidth, 
      height =\textwidth,
      xlabel ={$K$},
      ylabel ={Objective function $J(y,u)$},
      xmin=0,xmax=405, 
      title = {$\lambda = 10^{-3}$},
      ]
    	\addplot[style=densely dashdotdotted,
    	         color=blue,
    	         mark=diamond,
    	         mark options={solid},
    	         line width=1pt] table [x=K,y=J_SD]{\fnameone};  
	\addlegendentry{sd, $n_h=54$}
    	\addplot[style=densely dashdotdotted,
    	         color=red,
    	         mark=o,
    	         mark options={solid},
    	         line width=1pt] table [x=K,y=J_ST]{\fnameone};  	
	\addlegendentry{st, $n_h=54$}

    	\addplot[style=solid,
    	         color=blue,
    	         mark=diamond,
    	         mark options={solid},
    	         line width=1pt] table [x=K,y=J_SD]{\fnametwo};  	
	\addlegendentry{sd, $n_h=213$}
    	\addplot[style=solid,
    	         color=red,
    	         mark=o,
    	         mark options={solid},
    	         line width=1pt] table [x=K,y=J_ST]{\fnametwo};  	
 	\addlegendentry{st, $n_h=213$}
  \end{axis}
  
\end{tikzpicture}
\end{minipage}
	\caption{2d example, values of the objective function for different discretizations (left: $\lambda= 10^{-1}$, right: $\lambda=10^{-3}$, abbreviations: semi-discrete (sd), space-time (st)).}
	\label{plot_2D_obj_func}
\end{figure}

\begin{remark}\label{Rem:CPU}
	With the chosen all-at-once approach, we get very similar CPU times for both methods. As already pointed out earlier, a runtime comparison of best possible schemes is not the aim of this paper. Not using efficient tensorproduct solvers yields that the limiting factor is the memory -- in both cases. %added Not AR
\end{remark}

}
%====================================================================
\section{Summary, conclusions and outlook}\label{Sec:6}
%====================================================================
We have considered a space-time variational formulation for a PDE-cons\-trai\-ned optimal control problem. Our first-optimize-then-discretize approach follows the abstract functional analytic framework of such problems, which is then detailed for the space-time variational method. This can be summarized as follows:
\begin{compactitem}
	\item Well-posed space-time variational formulation of the state equation. This yields different trial and test spaces (Petrov--Galerkin style) of minimal regularity;
	\item Formulation of the optimal control problem in the arising spaces, definition of the Lagrange function and derivation of KKT conditions. This yields the adjoint and gradient equations in natural spaces with minimal regularity requirements;
	\item Derivation of (necessary and sufficient) optimality conditions and optimality system (still in the infinite-dimensional setting);
	\item LBB-stable discretization of the optimality system. In special cases, this can be chosen to be equivalent to a Crank--Nicolson semi-discrete discretization, which allows quantitative numerical comparisons.
\end{compactitem}
Moreover, we reported on numerical experiments showing that space-time methods yield the same value of the objective function for significantly smaller number of unknowns. Since the CPU-times for the same number of unknowns turned out to be similar, this offers potential for significant speedup.

Topics for future research include control and state constraints, other types of PDEs for the constraints, improved schemes for solving the optimality system, adaptive discretization of the control, etc. Also efficient solvers that explicitly exploit the Kronecker structures of arising operators should be investigated. Finally, the above setting seems to be a very good starting point for investigating model reduction, e.g. \cite{k.urbana.t.patera2014}.

\section*{Acknowledgements}
We are grateful for Stefan Hain (Ulm), Michael Hinze (Koblenz), Davide Palitta (Bologna) and Stefan Volkwein (Konstanz) for fruitful discussions and very helpful remarks.
%====================================================================
\bibliography{BRU}


\begin{thebibliography}{35}
\ifx \bisbn   \undefined \def \bisbn  #1{ISBN #1}\fi
\ifx \binits  \undefined \def \binits#1{#1}\fi
\ifx \bauthor  \undefined \def \bauthor#1{#1}\fi
\ifx \batitle  \undefined \def \batitle#1{#1}\fi
\ifx \bjtitle  \undefined \def \bjtitle#1{#1}\fi
\ifx \bvolume  \undefined \def \bvolume#1{\textbf{#1}}\fi
\ifx \byear  \undefined \def \byear#1{#1}\fi
\ifx \bissue  \undefined \def \bissue#1{#1}\fi
\ifx \bfpage  \undefined \def \bfpage#1{#1}\fi
\ifx \blpage  \undefined \def \blpage #1{#1}\fi
\ifx \burl  \undefined \def \burl#1{\textsf{#1}}\fi
\ifx \doiurl  \undefined \def \doiurl#1{\url{https://doi.org/#1}}\fi
\ifx \betal  \undefined \def \betal{\textit{et al.}}\fi
\ifx \binstitute  \undefined \def \binstitute#1{#1}\fi
\ifx \binstitutionaled  \undefined \def \binstitutionaled#1{#1}\fi
\ifx \bctitle  \undefined \def \bctitle#1{#1}\fi
\ifx \beditor  \undefined \def \beditor#1{#1}\fi
\ifx \bpublisher  \undefined \def \bpublisher#1{#1}\fi
\ifx \bbtitle  \undefined \def \bbtitle#1{#1}\fi
\ifx \bedition  \undefined \def \bedition#1{#1}\fi
\ifx \bseriesno  \undefined \def \bseriesno#1{#1}\fi
\ifx \blocation  \undefined \def \blocation#1{#1}\fi
\ifx \bsertitle  \undefined \def \bsertitle#1{#1}\fi
\ifx \bsnm \undefined \def \bsnm#1{#1}\fi
\ifx \bsuffix \undefined \def \bsuffix#1{#1}\fi
\ifx \bparticle \undefined \def \bparticle#1{#1}\fi
\ifx \barticle \undefined \def \barticle#1{#1}\fi
\bibcommenthead
\ifx \bconfdate \undefined \def \bconfdate #1{#1}\fi
\ifx \botherref \undefined \def \botherref #1{#1}\fi
\ifx \url \undefined \def \url#1{\textsf{#1}}\fi
\ifx \bchapter \undefined \def \bchapter#1{#1}\fi
\ifx \bbook \undefined \def \bbook#1{#1}\fi
\ifx \bcomment \undefined \def \bcomment#1{#1}\fi
\ifx \oauthor \undefined \def \oauthor#1{#1}\fi
\ifx \citeauthoryear \undefined \def \citeauthoryear#1{#1}\fi
\ifx \endbibitem  \undefined \def \endbibitem {}\fi
\ifx \bconflocation  \undefined \def \bconflocation#1{#1}\fi
\ifx \arxivurl  \undefined \def \arxivurl#1{\textsf{#1}}\fi
\csname PreBibitemsHook\endcsname

%%% 1
\bibitem{f.troeltzsch2009}
\begin{bbook}
\bauthor{\bsnm{Tr{\"o}ltzsch}, \binits{F.}}:
\bbtitle{Optimal Control of Partial Differential Equations: Theory, Methods,
  and Applications}.
\bpublisher{American Math. Soc.},
\blocation{Providence}
(\byear{2010})
\end{bbook}
\endbibitem

%%% 2
\bibitem{MR2516528}
\begin{bbook}
\bauthor{\bsnm{Hinze}, \binits{M.}},
\bauthor{\bsnm{Pinnau}, \binits{R.}},
\bauthor{\bsnm{Ulbrich}, \binits{M.}},
\bauthor{\bsnm{Ulbrich}, \binits{S.}}:
\bbtitle{Optimization with {PDE} Constraints}.
\bsertitle{Mathematical Modelling: Theory and Applications},
vol. \bseriesno{23}.
\bpublisher{Springer},
\blocation{Heidelberg}
(\byear{2009})
\end{bbook}
\endbibitem

%%% 3
\bibitem{MR2407012}
\begin{barticle}
\bauthor{\bsnm{Meidner}, \binits{D.}},
\bauthor{\bsnm{Vexler}, \binits{B.}}:
\batitle{A priori error estimates for space-time finite element discretization
  of parabolic optimal control problems. {I}. {P}roblems without control
  constraints}.
\bjtitle{SIAM J. Control Optim.}
\bvolume{47}(\bissue{3}),
\bfpage{1150}--\blpage{1177}
(\byear{2008})
\end{barticle}
\endbibitem

%%% 4
\bibitem{MR2874969}
\begin{barticle}
\bauthor{\bsnm{Neitzel}, \binits{I.}},
\bauthor{\bsnm{Vexler}, \binits{B.}}:
\batitle{A priori error estimates for space-time finite element discretization
  of semilinear parabolic optimal control problems}.
\bjtitle{Numer. Math.}
\bvolume{120}(\bissue{2}),
\bfpage{345}--\blpage{386}
(\byear{2012})
\end{barticle}
\endbibitem

%%% 5
\bibitem{MR2861431}
\begin{barticle}
\bauthor{\bsnm{Meidner}, \binits{D.}},
\bauthor{\bsnm{Vexler}, \binits{B.}}:
\batitle{A priori error analysis of the {P}etrov-{G}alerkin {C}rank-{N}icolson
  scheme for parabolic optimal control problems}.
\bjtitle{SIAM J. Contr. Opt.}
\bvolume{49}(\bissue{5}),
\bfpage{2183}--\blpage{2211}
(\byear{2011})
\end{barticle}
\endbibitem

%%% 6
\bibitem{MR3343358}
\begin{barticle}
\bauthor{\bparticle{von} \bsnm{Daniels}, \binits{N.}},
\bauthor{\bsnm{Hinze}, \binits{M.}},
\bauthor{\bsnm{Vierling}, \binits{M.}}:
\batitle{Crank-{N}icolson time stepping and variational discretization of
  control-constrained parabolic optimal control problems}.
\bjtitle{SIAM J. Contr. Opt.}
\bvolume{53}(\bissue{3}),
\bfpage{1182}--\blpage{1198}
(\byear{2015})
\end{barticle}
\endbibitem

%%% 7
\bibitem{MR4076464}
\begin{barticle}
\bauthor{\bparticle{von} \bsnm{Daniels}, \binits{N.}},
\bauthor{\bsnm{Hinze}, \binits{M.}}:
\batitle{Variational discretization of a control-constrained parabolic
  bang-bang optimal control problem}.
\bjtitle{J. Comput. Math.}
\bvolume{38}(\bissue{1}),
\bfpage{14}--\blpage{40}
(\byear{2020})
\end{barticle}
\endbibitem

%%% 8
\bibitem{MR4223221}
\begin{barticle}
\bauthor{\bsnm{Langer}, \binits{U.}},
\bauthor{\bsnm{Steinbach}, \binits{O.}},
\bauthor{\bsnm{Tr\"{o}ltzsch}, \binits{F.}},
\bauthor{\bsnm{Yang}, \binits{H.}}:
\batitle{Unstructured space-time finite element methods for optimal control of
  parabolic equations}.
\bjtitle{SIAM J. Sci. Comput.}
\bvolume{43}(\bissue{2}),
\bfpage{744}--\blpage{771}
(\byear{2021})
\end{barticle}
\endbibitem

%%% 9
\bibitem{j.l.lions1971}
\begin{bbook}
\bauthor{\bsnm{Lions}, \binits{J.L.}}:
\bbtitle{Optimal Control of Systems Governed by Partial Differential
  Equations}.
\bpublisher{Springer},
\blocation{New York}
(\byear{1971})
\end{bbook}
\endbibitem

%%% 10
\bibitem{lions.magenes.2}
\begin{bbook}
\bauthor{\bsnm{Lions}, \binits{J.L.}},
\bauthor{\bsnm{Magenes}, \binits{E.}}:
\bbtitle{Non-Homogeneous Boundary Value Problems and Applications}
vol. \bseriesno{2}.
\bpublisher{Springer},
\blocation{New York}
(\byear{1972})
\end{bbook}
\endbibitem

%%% 11
\bibitem{r.dautrayj.lions1992}
\begin{bbook}
\bauthor{\bsnm{Dautray}, \binits{R.}},
\bauthor{\bsnm{Lions}, \binits{J.}}:
\bbtitle{Mathematical Analysis and Numerical Methods for Science and
  Technology}.
\bpublisher{Springer},
\blocation{New York}
(\byear{1992})
\end{bbook}
\endbibitem

%%% 12
\bibitem{c.schwabr.stevenson2009}
\begin{barticle}
\bauthor{\bsnm{Schwab}, \binits{C.}},
\bauthor{\bsnm{Stevenson}, \binits{R.}}:
\batitle{Space-time adaptive wavelet methods for parabolic evolution problems}.
\bjtitle{Math. Comp.}
\bvolume{78}(\bissue{267}),
\bfpage{1293}--\blpage{1318}
(\byear{2009})
\end{barticle}
\endbibitem

%%% 13
\bibitem{k.urbana.t.patera2012}
\begin{barticle}
\bauthor{\bsnm{Urban}, \binits{K.}},
\bauthor{\bsnm{Patera}, \binits{A.}}:
\batitle{A new error bound for reduced basis approximation of parabolic partial
  differential equations}.
\bjtitle{C.R. Math. Acad. Sci. Paris}
\bvolume{3-4}(\bissue{350}),
\bfpage{203}--\blpage{207}
(\byear{2012})
\end{barticle}
\endbibitem

%%% 14
\bibitem{MR1971217}
\begin{barticle}
\bauthor{\bsnm{Xu}, \binits{J.}},
\bauthor{\bsnm{Zikatanov}, \binits{L.}}:
\batitle{Some observations on {B}abu\v ska and {B}rezzi theories}.
\bjtitle{Numer. Math.}
\bvolume{94}(\bissue{1}),
\bfpage{195}--\blpage{202}
(\byear{2003})
\end{barticle}
\endbibitem

%%% 15
\bibitem{r.andreev2012}
\begin{botherref}
\oauthor{\bsnm{Andreev}, \binits{R.}}:
Stability of space-time petrov-galerkin discretizations for parabolic evolution
  equations.
PhD thesis,
ETH Z{\"u}rich, Nr. 20842
(2012)
\end{botherref}
\endbibitem

%%% 16
\bibitem{k.urbana.t.patera2014}
\begin{barticle}
\bauthor{\bsnm{Urban}, \binits{K.}},
\bauthor{\bsnm{Patera}, \binits{A.}}:
\batitle{An improved error bound for reduced basis approximation of linear
  parabolic problems}.
\bjtitle{Math. Comp.}
\bvolume{83}(\bissue{288}),
\bfpage{1599}--\blpage{1615}
(\byear{2014})
\end{barticle}
\endbibitem

%%% 17
\bibitem{GunzburgerKunoth2011}
\begin{barticle}
\bauthor{\bsnm{Gunzburger}, \binits{M.D.}},
\bauthor{\bsnm{Kunoth}, \binits{A.}}:
\batitle{Space-time adaptive wavelet methods for optimal control problems
  constrained by parabolic evolution equations}.
\bjtitle{SIAM J. Contr. Opt.}
\bvolume{49}(\bissue{3}),
\bfpage{1150}--\blpage{1170}
(\byear{2011})
\end{barticle}
\endbibitem

%%% 18
\bibitem{stevenson2021waveletintime}
\begin{barticle}
\bauthor{\bsnm{Stevenson}, \binits{R.}},
\bauthor{\bparticle{van} \bsnm{Veneti{\"e}}, \binits{R.}},
\bauthor{\bsnm{Westerdiep}, \binits{J.}}:
\batitle{A wavelet-in-time, finite element-in-space adaptive method for
  parabolic evolution equations}.
\bjtitle{Adv. Comp. Math.}
\bvolume{48}(\bissue{3}),
\bfpage{17}
(\byear{2022})
\end{barticle}
\endbibitem

%%% 19
\bibitem{o.steinbach2015}
\begin{barticle}
\bauthor{\bsnm{Steinbach}, \binits{O.}}:
\batitle{Space-time finite element methods for parabolic problems}.
\bjtitle{Comp. Meth. Appl. Math.}
\bvolume{15}(\bissue{4}),
\bfpage{551}--\blpage{566}
(\byear{2015})
\end{barticle}
\endbibitem

%%% 20
\bibitem{j.henning.etal2019}
\begin{bchapter}
\bauthor{\bsnm{Henning}, \binits{J.}},
\bauthor{\bsnm{Palitta}, \binits{D.}},
\bauthor{\bsnm{Simoncini}, \binits{V.}},
\bauthor{\bsnm{Urban}, \binits{K.}}:
\bctitle{Matrix oriented reduction of space-time {Petrov-Galerkin} variational
  problems}.
In: \beditor{\bsnm{Vermolen}, \binits{F.J.}},
\beditor{\bsnm{Vuik}, \binits{C.}} (eds.)
\bbtitle{Numerical Mathematics and Advanced Applications {ENUMATH} 2019},
pp. \bfpage{1049}--\blpage{1057}.
\bpublisher{Springer},
\blocation{Switzerland}
(\byear{2019})
\end{bchapter}
\endbibitem

%%% 21
\bibitem{palitta2019matrix}
\begin{barticle}
\bauthor{\bsnm{Palitta}, \binits{D.}}:
\batitle{Matrix equation techniques for certain evolutionary partial
  differential equations}.
\bjtitle{J. Sci. Comput.}
\bvolume{3},
\bfpage{87}--\blpage{99}
(\byear{2021})
\end{barticle}
\endbibitem

%%% 22
\bibitem{TDCM14}
\begin{botherref}
\oauthor{\bsnm{Ellis}, \binits{T.E.}},
\oauthor{\bsnm{Demkowicz}, \binits{L.}},
\oauthor{\bsnm{Chan}, \binits{J.L.}},
\oauthor{\bsnm{Moser}, \binits{R.D.}}:
{Space-Time DPG: Designing a Method for Massively Parallel CFD}.
{ICES Report 14-32, Univ. Texas at Austin}
(2014)
\end{botherref}
\endbibitem

%%% 23
\bibitem{DemGop11}
\begin{barticle}
\bauthor{\bsnm{Demkowicz}, \binits{L.}},
\bauthor{\bsnm{Gopalakrishnan}, \binits{J.}}:
\batitle{A class of discontinuous {P}etrov-{G}alerkin methods. {II}. {O}ptimal
  test functions}.
\bjtitle{Numer. Meth. PDEs}
\bvolume{27}(\bissue{1}),
\bfpage{70}--\blpage{105}
(\byear{2011})
\end{barticle}
\endbibitem

%%% 24
\bibitem{r.andreev2016A}
\begin{barticle}
\bauthor{\bsnm{Andreev}, \binits{R.}}:
\batitle{On long time integration of the heat equation}.
\bjtitle{Calcolo}
\bvolume{53}(\bissue{1}),
\bfpage{19}--\blpage{34}
(\byear{2016})
\end{barticle}
\endbibitem

%%% 25
\bibitem{Yano14}
\begin{barticle}
\bauthor{\bsnm{Yano}, \binits{M.}}:
\batitle{A space-time {P}etrov--{G}alerkin certified reduced basis method:
  Application to the {B}oussinesq equations}.
\bjtitle{SIAM J. Sci. Comput.}
\bvolume{36}(\bissue{1}),
\bfpage{232}--\blpage{266}
(\byear{2014})
\end{barticle}
\endbibitem

%%% 26
\bibitem{c.mollet2016}
\begin{botherref}
\oauthor{\bsnm{Mollet}, \binits{C.}}:
{Parabolic PDEs in Space-Time Formulations \textendash \ Stability for
  Petrov-Galerkin Discretizations with B-Splines and Existence of Moments for
  Problems with Random Coefficients}.
PhD thesis,
Univ. K{\"o}ln
(2016)
\end{botherref}
\endbibitem

%%% 27
\bibitem{MR3328986}
\begin{bbook}
\bauthor{\bsnm{Hinze}, \binits{M.}},
\bauthor{\bsnm{K\"{o}ster}, \binits{M.}},
\bauthor{\bsnm{Turek}, \binits{S.}}:
In: \beditor{\bsnm{Leugering}, \binits{G.}},
\beditor{\bsnm{Benner}, \binits{P.}},
\beditor{\bsnm{Engell}, \binits{S.}},
\beditor{\bsnm{Griewank}, \binits{A.}},
\beditor{\bsnm{Harbrecht}, \binits{H.}},
\beditor{\bsnm{Hinze}, \binits{M.}},
\beditor{\bsnm{Rannacher}, \binits{R.}},
\beditor{\bsnm{Ulbrich}, \binits{S.}} (eds.)
\bbtitle{Space-time {N}ewton-multigrid strategies for nonstationary distributed
  and boundary flow control problems},
pp. \bfpage{383}--\blpage{401}.
\bpublisher{Springer},
\blocation{Cham}
(\byear{2014})
\end{bbook}
\endbibitem

%%% 28
\bibitem{MR2872584}
\begin{barticle}
\bauthor{\bsnm{Borz\`\i}, \binits{A.}},
\bauthor{\bsnm{Gonz\'{a}lez~Andrade}, \binits{S.}}:
\batitle{Multigrid solution of a {L}avrentiev-regularized state-constrained
  parabolic control problem}.
\bjtitle{Numer. Math. Theory Methods Appl.}
\bvolume{5}(\bissue{1}),
\bfpage{1}--\blpage{18}
(\byear{2012})
\end{barticle}
\endbibitem

%%% 29
\bibitem{j.delosreyes2015}
\begin{bbook}
\bauthor{\bparticle{los} \bsnm{Reyes}, \binits{J.D.}}:
\bbtitle{Numerical PDE-Constrained Optimization}.
\bpublisher{Springer}, \blocation{???}
(\byear{2015})
\end{bbook}
\endbibitem

%%% 30
\bibitem{m.hinzer.pinnaum.ulbrichs.ulbrich2009}
\begin{bbook}
\bauthor{\bsnm{Hinze}, \binits{M.}},
\bauthor{\bsnm{Pinnau}, \binits{R.}},
\bauthor{\bsnm{Ulbrich}, \binits{M.}},
\bauthor{\bsnm{Ulbrich}, \binits{S.}}:
\bbtitle{Optimization with PDE Constraints}.
\bpublisher{Springer},
\blocation{Heidelberg}
(\byear{2009})
\end{bbook}
\endbibitem

%%% 31
\bibitem{r.andreev2013A}
\begin{barticle}
\bauthor{\bsnm{Andreev}, \binits{R.}}:
\batitle{Stability of sparse space-time finite element discretizations of
  linear parabolic evolution equations}.
\bjtitle{IMA J. Numer. Anal.}
\bvolume{33}(\bissue{1}),
\bfpage{242}--\blpage{260}
(\byear{2013})
\end{barticle}
\endbibitem

%%% 32
\bibitem{Hinze2005}
\begin{barticle}
\bauthor{\bsnm{Hinze}, \binits{M.}}:
\batitle{A variational discretization concept in control constrained
  optimization: The linear-quadratic case}.
\bjtitle{Comp. Opt. Appl.}
\bvolume{30}(\bissue{1}),
\bfpage{45}--\blpage{61}
(\byear{2005})
\end{barticle}
\endbibitem

%%% 33
\bibitem{HeRoSt16}
\begin{bbook}
\bauthor{\bsnm{Hesthaven}, \binits{J.S.}},
\bauthor{\bsnm{Rozza}, \binits{G.}},
\bauthor{\bsnm{Stamm}, \binits{B.}}:
\bbtitle{Certified Reduced Basis Methods for Parametrized Partial Differential
  Equations}.
\bpublisher{Springer},
\blocation{Cham}
(\byear{2016})
\end{bbook}
\endbibitem

%%% 34
\bibitem{s.glasa.mayerhoferk.urban2017}
\begin{bchapter}
\bauthor{\bsnm{Glas}, \binits{S.}},
\bauthor{\bsnm{Mayerhofer}, \binits{A.}},
\bauthor{\bsnm{Urban}, \binits{K.}}:
\bctitle{Two ways to treat time in reduced basis methods}.
In: \beditor{\bsnm{Benner}, \binits{P.}},
\beditor{\bsnm{Ohlberger}, \binits{M.}},
\beditor{\bsnm{Patera}, \binits{A.}},
\beditor{\bsnm{Rozza}, \binits{G.}},
\beditor{\bsnm{Urban}, \binits{K.}} (eds.)
\bbtitle{Model Reduction of Parametrized Systems},
pp. \bfpage{1}--\blpage{16}.
\bpublisher{Springer},
\blocation{Cham}
(\byear{2017})
\end{bchapter}
\endbibitem

%%% 35
\bibitem{henning2021weak}
\begin{barticle}
\bauthor{\bsnm{Henning}, \binits{J.}},
\bauthor{\bsnm{Palitta}, \binits{D.}},
\bauthor{\bsnm{Simoncini}, \binits{V.}},
\bauthor{\bsnm{Urban}, \binits{K.}}:
\batitle{An ultraweak space-time variational formulation for the wave equation:
  Analysis and efficient numerical solution}.
\bjtitle{ESAIM: M2AN}
\bvolume{56}(\bissue{4}),
\bfpage{1173}--\blpage{1198}
(\byear{2022})
\end{barticle}
\endbibitem

\end{thebibliography}
%====================================================================

\end{document}